\documentclass{siamart}

\usepackage{lipsum}
\usepackage{amsfonts}
\usepackage{graphicx}
\usepackage{epstopdf}
\usepackage{algorithmic}
\ifpdf
  \DeclareGraphicsExtensions{.eps,.pdf,.png,.jpg}
\else
  \DeclareGraphicsExtensions{.eps}
\fi

\newcommand{\TheTitle}{Parametric interpolation with the equal area principle} 
\newcommand{\TheAuthors}{G. McGregor, J.-C. Nave}

\headers{\TheTitle}{\TheAuthors}

\title{{A parametric interpolation framework for 1D scalar conservation laws using the equal area principle}\thanks{The research of GMc was supported in part by a Schulich Scholarship at McGill University. The research of JCN was supported in part by the NSERC Canada Discovery Grants Program. Additionally, JCN would like to thank the Shanghai Jiaotong University Institute of Natural Sciences for hosting him while completing this work.}}

\author{
  Geoffrey McGregor\thanks{Department of Mathematics, McGill University, Montreal, QC, Canada
    (\email{Geoffrey.McGregor@mail.mcgill.ca}}
  \and
  Jean-Christophe Nave\thanks{Department of Mathematics, McGill University, Montreal, QC, Canada
    (\email{jcnave@math.mcgill.ca}}
}

\usepackage{amsopn}

\newtheorem{alg}{Algorithm}
\newtheorem{cor}[alg]{Corollary}

\newtheorem{lem}{Lemma}

\newtheorem{rem}{Remark}
\newtheorem{thm}{Theorem}
\newtheorem{dfn}{Definition}

\ifpdf
\hypersetup{
  pdftitle={\TheTitle},
  pdfauthor={\TheAuthors}
}
\fi

\begin{document}

\maketitle

% REQUIRED
\begin{abstract}
In this paper we develop a novel framework for numerically solving scalar conservation laws in one space dimension. Utilizing the method of characteristics in conjunction with the equal area principle we develop an approach where the weak solution is obtained purely as the solution of a parametric interpolation problem.  As this framework hinges on the validity of the equal area principle, we provide a rigorous discussion of the equal area principle and show that, indeed, the equal area principle is equivalent to the Rankine-Hugoniot condition, within the specific context studied in this paper. Combining these results with properties of the characteristic equations yields the desired setting to define the equivalent parametric interpolation problem. We conclude by applying this framework to Burgers' equation and show how one obtains machine precision in the shock position when the initial condition can be represented exactly in the chosen space of parametric polynomials.
\end{abstract}

% REQUIRED

% REQUIRED

\section{Introduction}
In this paper we study the Cauchy problem
\begin{equation}
\begin{cases}
u_t+\left(F(u)\right)_{x}=0, \quad \text{for $x,t \in \mathbb{R}$, with $t\geq0$ and $F\in C^2(\mathbb{R})$ uniformly convex}\\ \label{PDE}
u(x,0)=g(x),
\end{cases}
\end{equation}
with a focus on its characteristic equations
\begin{align}
\dot{x}&=F'(u)\nonumber\\ \label{Char1}
\dot{u}&=0.
\end{align}
In particular, we investigate an alternative set of equations which we refer to as the characteristic flow,
\begin{align}
x(x_0,t)&=x_0+F'(g(x_0))t\label{CharFlow1}\\
u(x,t)&=g(x_0).\nonumber
\end{align}
In the pursuit of weak numerical solutions of (\ref{PDE}), there are numerous advantages to studying (\ref{CharFlow1}) instead of (\ref{Char1}). First, we note that the solutions of (\ref{CharFlow1}) exist for all time, whereas the solutions of (\ref{Char1}) only exist until shocks form, which then requires the explicit use of the Rankine-Hugoniot condition, see \cite{Hugo} and \cite{Rankine}. Alternatively, weak solutions of (\ref{PDE}) can be obtained from (\ref{CharFlow1}) by applying the equal area principle in place of the Rankine-Hugoniot condition, see \cite{Whit} and \cite{Equal}. Specifically, equation (\ref{CharFlow1}) is solved beyond when shocks form, creating a multivalued curve.  The desired solution is then recovered by a particular area-preserving projection, namely, the equal area principle. This paper is devoted to rigorously defining such a projection, and through its use, providing an appropriate framework for high-order numerical computation.

\begin{figure}[!ht]
\begin{center}
\includegraphics[width=40mm,height=30mm]{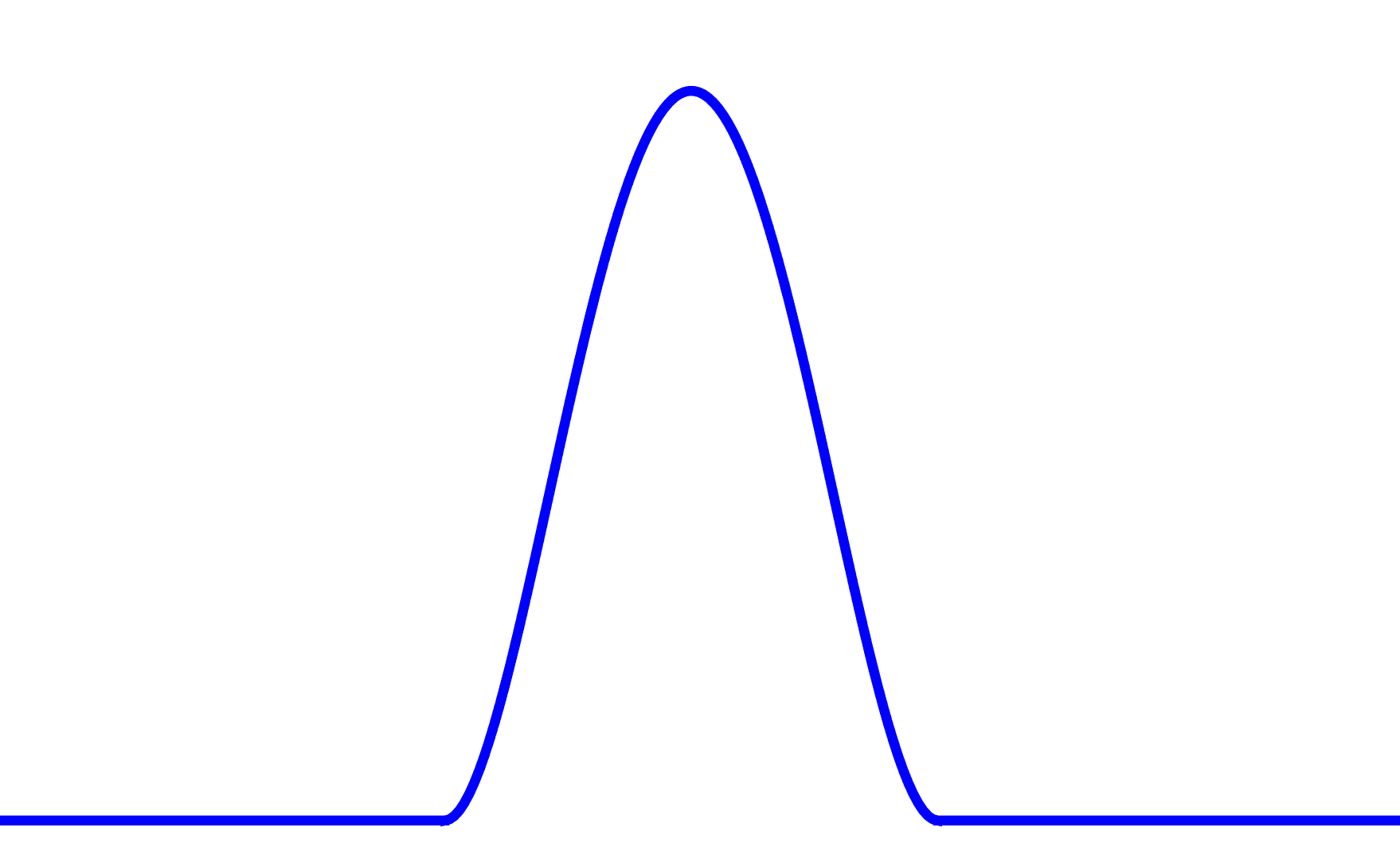}
\includegraphics[width=40mm,height=30mm]{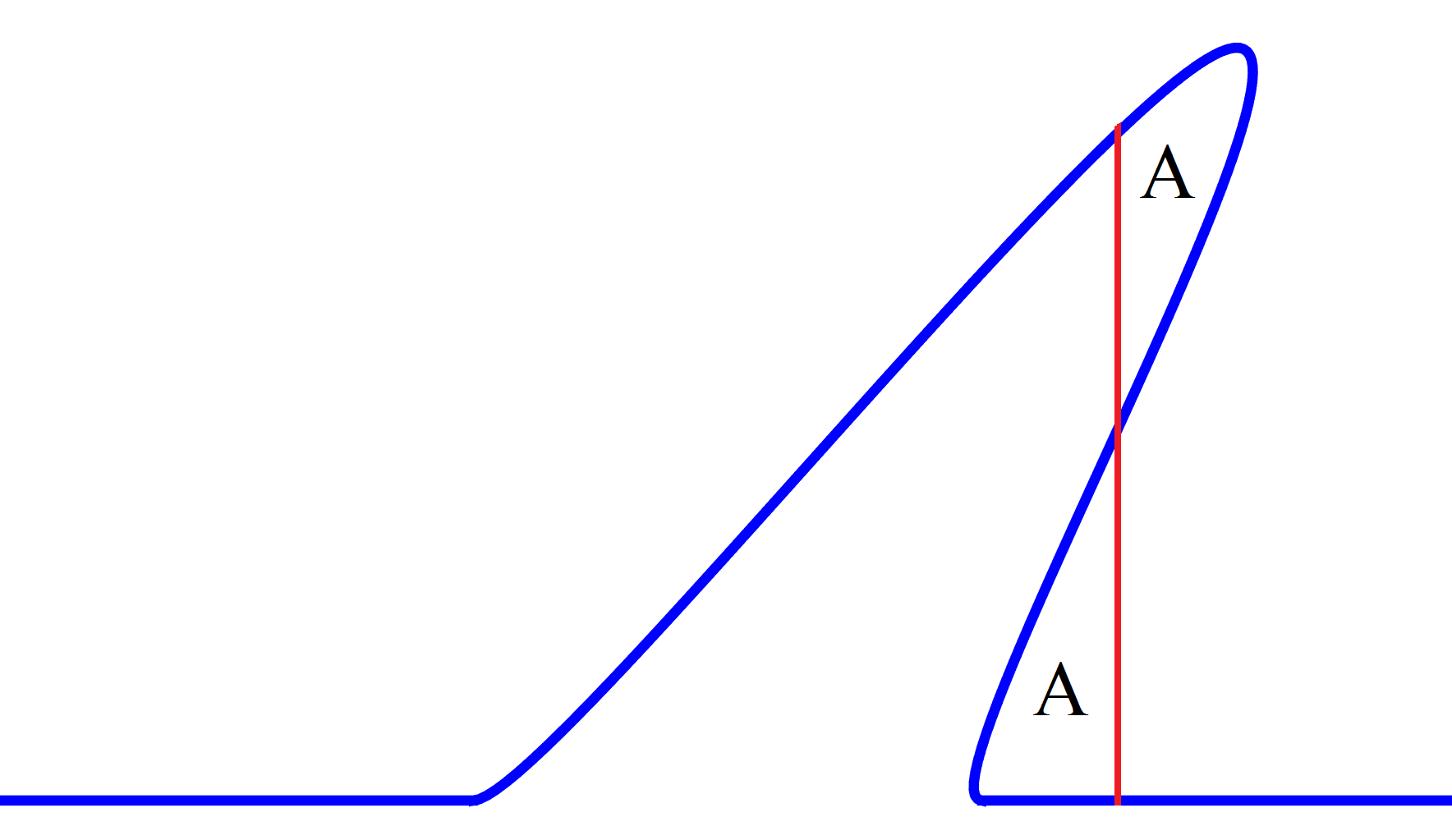}
\includegraphics[width=40mm,height=30mm]{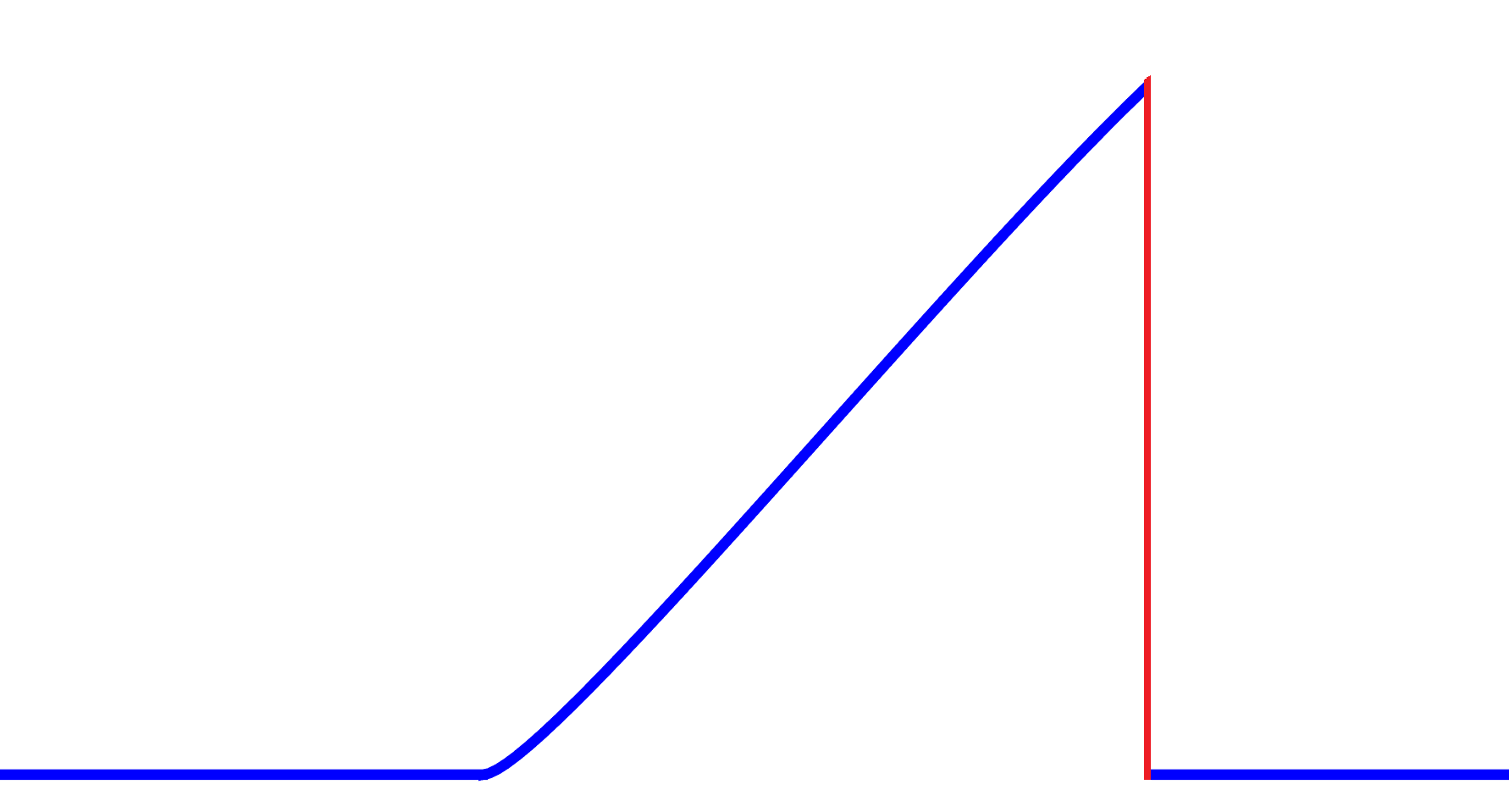}
\end{center}
\caption{ An illustration of the equal area principle. }
\label{LeVeque Plot}
\end{figure} 

We begin understanding how the equal area principle can be an asset for numerical methods by studying Figure \ref{LeVeque Plot} and by viewing (\ref{CharFlow1}) as the parametric curve
\begin{equation}
(x(x_0,t),u(x_0))=(x_0+F'(g(x_0))t,g(x_0)). \label{flow}
\end{equation}
First observe that the time evolution of (\ref{flow}) appears to preserve geometric smoothness.  We will show that this is indeed the case, provided $F$ and $g$ possess sufficient regularity. As discussed in \cite{parsmooth}, a parametric curve has geometric smoothness of order $n$ provided each component has $n$ continuous derivatives and that the velocity is never $\vec{0}$. It is easy to see that both components of (\ref{flow}) have $n$ continuous derivatives provided $F\in C^{n+1}$ and $g\in C^n$, and the velocity vector $(1+F''(g(x_0))g'(x_0)t,g'(x_0))\neq0$ for any $x_0$. Therefore, (\ref{flow}) indeed preserves smoothness as time evolves, provided $F$ and $g$ are sufficiently smooth.  Second we see from Figure \ref{LeVeque Plot} that the shock position is located by finding regions of equal area,  labeled by $A$. Therefore, computing the shock location numerically will require accurate integration of these two regions.  These two observations form the basis of our proposed framework. If we allow the curve to overturn and remain smooth, we maintain our ability to use high-order parametric interpolation which lends itself to accurate integration and thus an accurate representation of the weak solution. It is from this perspective that we see how an equivalent parametric interpolation problem for solving (\ref{PDE}) can be formulated. Before presenting the details of our numerical framework, we first discuss other numerical methods which can be used to solve (\ref{PDE}).

Due to a lack of rigorous treatment there are currently very few numerical methods which utilize the equal area principle. There are, however, many other methods which have proven useful for solving (\ref{PDE}) numerically. For example: finite volume methods, \cite{Spectral}, \cite{eno}, \cite{Maximum} and \cite{LevFinite}, discontinuous galerkin, \cite{ShuRunge}, particle methods, \cite{Seibold} , and high resolution finite difference techniques, \cite{Highres}, \cite{Maximum} and \cite{Highres2}, to just name a few. There are desirable properties that these methods aim to possess, such as having high accuracy near shocks with sharp features, preserving a maximum principle, being conservative or having scalability to higher order, all while staying computationally efficient. A particularly relevant example is the work of Seibold and Farjoun in \cite{Seibold}. Their approach utilizes the equal area principle with linear interpolation to create a method which obtains second order accuracy in the shock position. Although the emphasis of their work is on particle management, its core philosophy inspired much of what is presented here. This leads us to discuss further details of our proposed numerical framework.

We begin with the crucial observation that the two curves in Figure \ref{LeVeque Plot} are given exactly by (\ref{flow}). Also, the curve, its tangents and higher order derivatives along with anti-derivatives are given purely in terms of $F$ and $g$. Therefore, we can use the exact data from (\ref{flow}) to construct a parametric polynomial representation of the solution at any time $t$. This, however, requires that the weak solution can always be obtained from (\ref{flow}), meaning, the equal area principle must holds for any initial condition $g$ and flux function $F$ in the desired class. Therefore, before we can properly discuss the interpolation framework, we must study the equal area principle in more detail. 

The equal area principle, or equal area property, is well known within the hyperbolic conservation law community, with several notable contributions reinforcing its validity. For example, LeVeque in \cite{LeVeque} introduces the equal area principle as a useful tool for finding the shock position. In \cite{Equal} the authors prove a local equivalence between the Rankine-Hugoniot condition and the equal area principle. The local equivalence proof in \cite{Equal}, however, relies on the implicit function theorem with no clear extension argument available, therefore this result does not provide insight into the long term validity of the equal area principle. In \cite{Whit}, Whitham proves that the equal area principle produces the desired shock speed when solving a Riemann problem with piecewise constant states. This result holds for all uniformly convex flux function $F$ and for all time, but not for general initial conditions $g(x)$.  Therefore, before numerical methods which employ the equal area principle can be relied upon, further analytical justification is required.

%%For the majority of this paper we will view equation (\ref{CharFlow1}) as the parametric curve
%\begin{equation}
%(x(x_0,t),u(x_0))=(x_0+F'(g(x_0))t,g(x_0)),
%\end{equation}
%parametrized by $x_0$. It is from this perspective that we see how an equivalent parametric interpolation problem can be formulated. 

 This brings us to the first result of our paper; we show that the equal area principle holds for all initial conditions $g(x)$ and all time $t$, provided the shock remains isolated. To achieve this, we introduce an extension of the equal area principle and prove its equivalence with the Rankine-Hugoniot condition. This extension will also prove useful for the development of numerical methods. Our second result pertains to interacting shocks. Here we prove that the equal area principle agrees with the Rankine-Hugoniot condition as well. With these results we can conclude that for all time $t$ the parametric curve (\ref{flow}) can be used to obtain the correct weak solution of (\ref{PDE}) by applying the equal area principle. Once proven, we have all of the required ingredients to construct numerical schemes which rely entirely on the parametric interpolation of (\ref{flow}).

Throughout this paper we rely on a few key analytical results associated with equation (\ref{Char1}).  Utilizing the theory of non-smooth systems and differential inclusion, developed by Fillipov \cite{Filli},  Dafermos \cite{Daf} made significant progress on the understanding of the solution structure of (\ref{Char1}). For example, Dafermos proved that characteristics can only propagate at their classical characteristic speed, $\dot{x}=F'(u)$, or at the speed determined by the Rankine-Hugoniot condition as discussed above.  Of particular relevance is Theorem 4.1 of \cite{Daf}, which shows that $S(t)$ is piecewise $C^1$ and that the only points where differentiability is lost occurs when there is interaction with another shock.  These results are pivotal in proving the equivalence between the Rankine-Hugoniot jump condition and the equal area principle.

This paper is organized as follows: first, in section \ref{Justif}, we work through a simple, yet non-trivial example justifying the equal area principle and then, in section \ref{global} we proceed to prove that the equal area principle is equivalent to the Rankine-Hugoniot condition for isolated shocks. In section \ref{GEAP} we present a more general form of equal area principle and use it to prove equivalence with the Rankine-Hugoniot condition when shocks interact. Finally in section \ref{Numerics} we lay the foundation for parametric interpolation methods which relies on our main results from sections \ref{global} and \ref{GEAP}. We then conclude with some numerical results and a discussion in section \ref{Disc}.
%we provide the foundation for novel numerical methods which rely on the theory presented in this paper. We then conclude with a discussion in section \ref{Disc}.

\section{Motivating Example}\label{Justif}

We begin by solving the Riemann problem (\ref{BurgerEx1}) first using the method of characteristics and employing the Rankine-Hugoniot condition, then again using the equal area principle. We show that the resulting equations for the shock position, $S(t)$, are the same. The following example does not have a piecewise constant initial condition, and we are seeking an equation for the shock position for all $t>0$. Therefore, the results provided in \cite{Equal} and \cite{Whit} do not account for this case. 

Consider the following simple Cauchy problem,
\begin{equation}
\begin{cases}
u_t+\left(\frac{u^2}{2}\right)_{x}=0,\quad \text{on $\mathbb{R}\times(0,\infty)$}\\       
u(x,0)=g(x)\quad \text{on $\mathbb{R}\times\{0\}$},\label{BurgerEx1}
\end{cases}
\end{equation}
where the initial condition $g(x)$ is defined by.
\begin{equation}g(x)=
\begin{cases}
0 \quad \text{for $x<0$} \label{IC1}\\
x \quad \text{for $0\leq x \leq 1$}\nonumber\\
0 \quad \text{for $x>1$}\nonumber
\end{cases}
\end{equation}

\begin{figure}[!ht]
\begin{center}
\includegraphics[width=80mm,height=40mm]{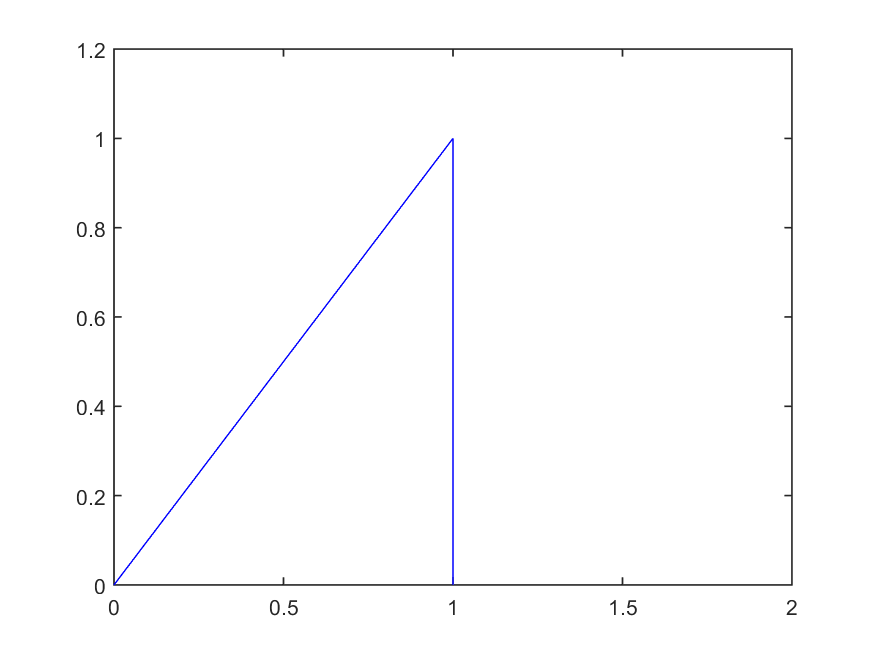}
\end{center}
\caption{A plot of the initial condition $g(x)$ from problem (\ref{BurgerEx1})}
\label{Ex1}
\end{figure} 

A basic application of the method of characteristics and utilizing the Rankine-Hugoniot condition gives the shock position at time $t$,
\begin{equation*}
S(t)=\sqrt{1+t}.
\end{equation*}

We now derive the same equation for the shock position by utilizing the equal area principle. The method of characteristics tells us that any point on the initial curve $(x(s,0), u(s))$ maps to $(x(s,0) + F^{\prime}(u(s))t,u(s))=(x(s,0) + u(s)t,u(s))$ at time $t$.  We define the Burgers' equation flow map $\Phi_B:\mathbb{R}^3\rightarrow\mathbb{R}^2$ by
\begin{equation}
\Phi_B (x(s,0),u(s),t)=(x(s,0) + u(s)t,u(s)),
\end{equation}
which is a well-defined map for all time $t>0$.  Next we find a suitable parametrization of the initial curve $(x(s,0),u(s))$, for example
\begin{equation}(x(s,0),u(s))=
\begin{cases}
(0,0) \quad \text{for $s < 0$}\\ \label{IC2}
(s,s) \quad \text{for $s \in [0,1]$}\\
(1,s-1) \quad \text{for $s \in [1,2]$}\\
(0,0) \quad \text{for $s > 2,$}
\end{cases}
\end{equation}
yielding, after time $t$, the parametrization 
\begin{equation}\Phi_B(x(s,0),u(s),t)=
\begin{cases}
(0,0) \quad \text{for $s < 0$}\\ \label{Map1}
(s+st,s) \quad \text{for $s \in [0,1]$}\\
(1+(s-1)t,s-1) \quad \text{for $s \in [1,2]$}\\
(0,0) \quad \text{for $s > 2$.}
\end{cases}
\end{equation}
The graph of $\Phi_B(x(s,0),u(s),1)$ is shown below in Figure \ref{Ex1Overturned}.

\begin{figure}[!ht]
\begin{center}
\includegraphics[width=63mm,height=40mm]{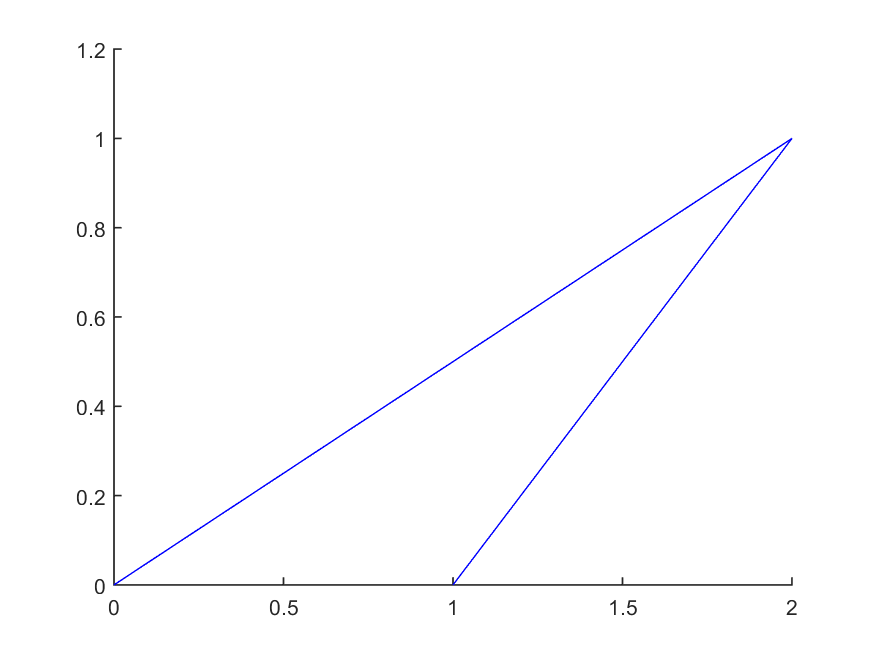}
\includegraphics[width=63mm,height=40mm]{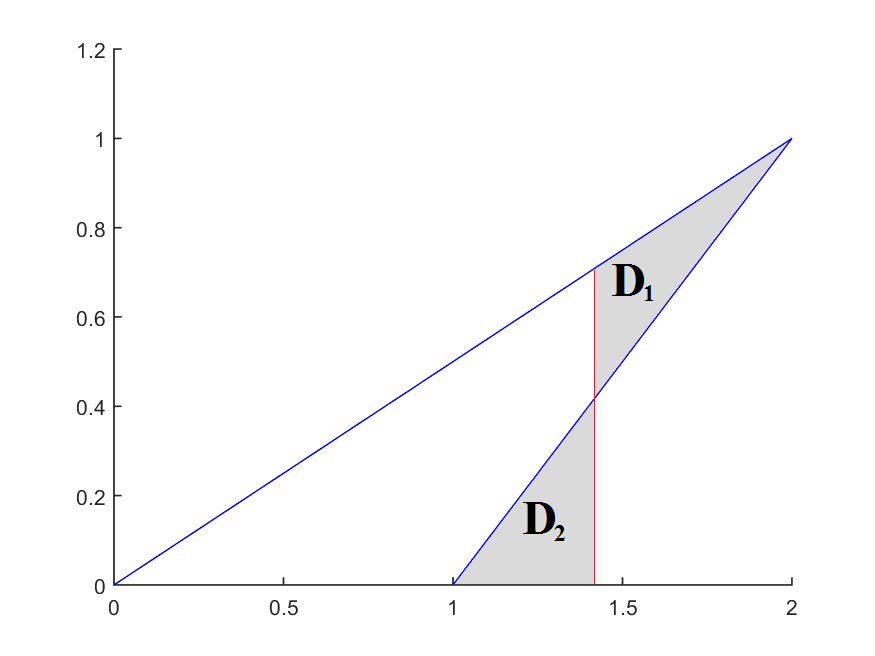}
\end{center}
\caption{A plot of $\Phi_{B}(x(s,0),g(s),1)$ and approximate equal area sketch.}
\label{Ex1Overturned}
\end{figure}

The equal area principle states that the shock location is given by the vertical line which splits the overturned curve into two regions with equal area. Referencing Figure \ref{Ex1Overturned}, this says that the red line is the shock location if the area of $D_1$ is equal to the area of $D_2$. We show the vertical line must be positioned at $S(t)=\sqrt{1+t}$, the same as our above computation given by the Rankine-Hugoniot condition.

%\begin{figure}[!ht]
%\begin{center}
%\includegraphics[width=80mm,height=40mm]{Ex1OverturnedShock}
%\end{center}
%\caption{The shock locations shown for $\Phi_{B}(s,g(s),1)$, where $g(x)$ is the initial condition from (\ref{BurgerEx1})}
%\label{Ex1Overturned}
%\end{figure} 

After time $t$ equation (\ref{Map1}) says that the top line in Figure \ref{Ex1Overturned} is given by the function $u(x)=\frac{x}{1+t}$ and that the lower line is given by $u(x)=\frac{x-1}{t}$. Requiring that $Area(D_1)$=$Area(D_2)$ gives us the equation
\begin{equation}
\int_{S(t)}^{1+t}{\frac{x}{1+t}-\frac{x-1}{t}dx}=\int_{1}^{S(t)}{\frac{x-1}{t}dx},
\end{equation}
where $S(t)$ denotes the shock position. An elementary computation leads us to $S^{2}(t)=1+t$ and since $u>0$ we take the positive root which results in $S(t)=\sqrt{1+t}$, as desired.

Next we consider a more generic overturned curve $(x(s,0),u(s))$, as done in \cite{Equal}, and extend the idea used in our motivating example above to the general flow map 
\begin{equation}
\Phi(x(s,0),u(s),t)=\left(x(s,0) + F^{\prime}(u(s))t,u(s)\right),\label{Map2}
\end{equation}
and prove that the vertical line giving zero signed area moves at Rankine-Hugoniot speed, provided the shock remains isolated. For notational purposes we write \\ $\Phi(x(s,0),u(s),t)=(x(s,t),u(s))$. We note that (\ref{Map2}) is more general than (\ref{flow}), as we allow the initial curve to be multivalued.
\begin{figure}[!ht]
\begin{center}
\includegraphics[width=100mm,height=70mm]{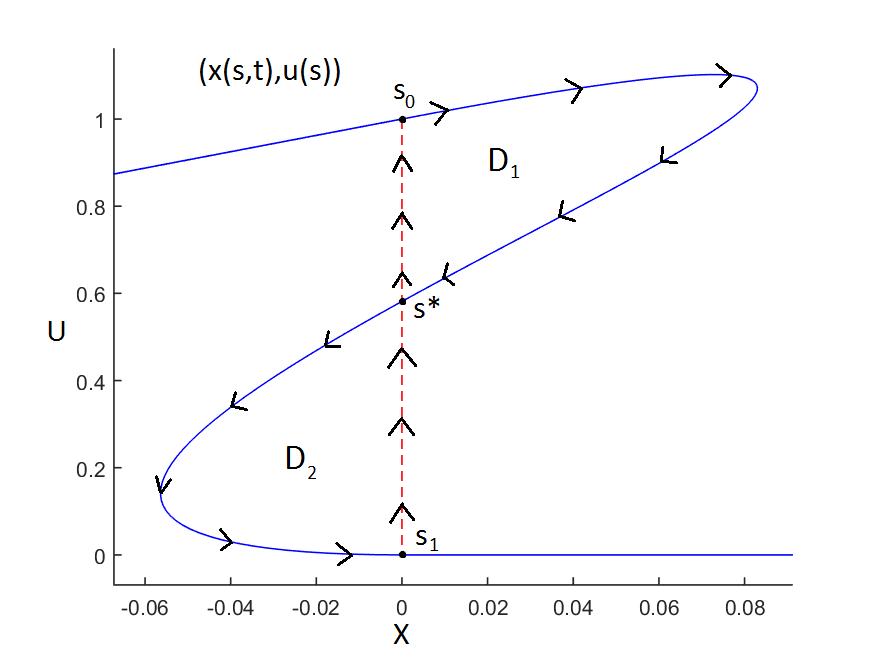}
\end{center}
\caption{Generic overturned parametric curve $(x(s,t), u(s))$}
\label{Overturned2}
\end{figure} 
\section{Global Equivalence}\label{global}
In this section we provide a proof that the equal area principle is equivalent to the Rankine-Hugoniot condition for any $g$ and $F$ given in (\ref{PDE}) for all time $t>0$, provided the shock remains isolated. More precisely,  let $(x(s,t),u(s))$ be a piecewise smooth parametrized curve flowing under (\ref{Map2}) and suppose it is overturned, that is, $t>t_s$, as in Figure \ref{Overturned2}. Then the vertical line creating equal area, defined below, moves at the speed given by the Rankine-Hugoniot condition for all time $t>t_s$.  

Before we begin with the proof, we discuss the shape of the curve \\$(s+F'(g(s))t,g(s))$, specifically in the overturned region. The Cauchy problem (\ref{PDE}) only considers uniformly convex flux function $F$. This implies that $F'(u)>0$ which results in its inverse, $F'^{-1}(u)$, being a well-defined increasing function. Now, suppose we have a shock at $x_0$ between states $u_L$ and $u_R$. We first parametrize the solution to the left of the shock by $(s+F'(g(s))t,g(s))$, for $s\in[0,s_0]$, and to the right of the shock with $(s+F'(g(s))t,g(s))$, for $s\in[s_1,1]$. We also parametrize the shock by $(x_0,z)$, for $z\in[u_R,u_L]$ and assume that $1+F''(g(s))g'(s)t_0>\alpha>0$, for some $\alpha>0$ and for all $s\in[0,s_0]\cup[s_1,1]$. Under the flow (\ref{flow}), the shock line gets mapped to $(x_0+F'(z)t,z)$, for $z\in[u_R,u_L]$, which can be written as the function
\begin{equation}
u(x)=F'^{-1}\left(\frac{x-x_0}{t}\right).\label{func}
\end{equation} 
Therefore, by our assumptions on $F$, we have that equation (\ref{func}) defines a function for all values of $t>t_0$, meaning any vertical line crosses $u(x)$ at most once. Also by our assumptions outside the shock we have that $1+F''(g(s))g'(s)t>0$ for all $t\in[t_0,t_0+\Delta t]$, for some $\Delta t>0$ small enough. This implies that $(s+F'(g(s))t,g(s))$ to the left and right of the shock (\ref{func}) remain functions for $t\in[t_0,t_0+\Delta t]$, and thus, each portion intersects any vertical line at most once. Therefore, in the overturned region of an isolated shock we expect to have exactly three intersections with the vertical shock line. This brings us to the definition of the an S-curve.

\begin{dfn}[S-curve] \label{EAC}
We say a parametrized curve $( x(s,t), u(s) )$ for $s \in \mathbb{R}$ is an S-curve at time $t$ if each of its overturned regions intersect any vertical line at exactly three points along the parametrization, $s_0<s^*<s_1$, with $x_s(s_0,t)>0, \, x_s(s_1,t)>0 $ and $x_s(s^*,t)<0$, such that $u(s_0)>u(s^*)>u(s_1)$. Additionally we require that $(x(s,t),u(s))$ does not cross itself, implying that the vertical line splits the region bounded by $s_0$ and $s_1$ into two distinct closed regions.  Furthermore, we say that $(x(s,t),u(s))$ is an S-curve on an interval of time $t\in(t_0,t_1)$ provided $(x(s,t),u(s))$ is an S-curve at each time $t$ in the given interval.
\end{dfn}

\begin{rem}
It follows by the uniform convexity of $F$ that the shape of the S-curve will be maintained, provided the shock remains isolated.
\end{rem}

We proceed to prove that the equal area principle applied to S-curves is equivalent to the Rankine-Hugoniot condition. We begin this discussion by referencing Figure \ref{Overturned2}. Let the first region, $D_1$, be defined by the portion of $\mathbb{R}^2$ enclosed by the parametrized curve from $s_0$ to $s^*$ and the line from $(x(s^*),u(s^*))$ to $(x(s_0),u(s_0))$. Green's theorem applied to a clockwise oriented close curve $(y(s),z(s))$ says the area enclosed within the curve is given by $\frac{1}{2}\int_{C}{yz'-y'zds}$. Since $D_1$ is enclosed by two curves, after some simplifying, Green's theorem says the area of $D_1$ is given by
\begin{equation}
A_{D_1}=\frac{1}{2}\int_{s_{0}}^{s^*}{x(s,t)u'(s)-x_s(s,t)u(s) \,\text{ds}}+\frac{1}{2}
\left((u(s_0)x(s^*,t)-x(s_0,t)u(s^*)\right)\label{AD1}
\end{equation}, and similarly the area of $D_2$ is given by
\begin{equation}
A_{D_2}=\frac{1}{2}\int_{s^{*}}^{s_1}{x_s(s,t)u(s)-x(s,t)u'(s) \,\text{ds}}+\frac{1}{2}
\left((x(s^*,t)u(s_1)-u(s^*)x(s_1,t)\right).\label{AD2}
\end{equation}
Combining (\ref{AD1}) and (\ref{AD2}) we have an equation for the difference between the two areas, referred to as the signed area, given as $A_{Dif}(s_0,s^*,s_1,t)=A_{D_1}-A_{D_2}$, yielding
\begin{align}
A_{Dif}(s_0,s^*,s_1,t)=&\frac{1}{2}\int_{s_{0}}^{s_1}{x(s,t)u'(s)-x_s(s,t)u(s) \,\text{ds}} \label{ADif} \\
&+\frac{1}{2}\left(x(s^*,t)(u(s_0)-u(s_1))+u(s^*)(x(s_1,t)-x(s_0,t))\right). \nonumber
\end{align}

%%\begin{rem}
%The fourth argument of $A_{Dif}(s_0,s^*,s_1,t)$ is time, therefore under the flow $\Phi$, from equation (\ref{Map2}), equation (\ref{ADif}) remains the same except that $x(s)$ would be replaced by $x(s,t)$. We also note the difference between equation (\ref{ADif}) and the signed area formula presented in \cite{Equal}, provided the shock remains isolated.
%\end{rem}

With these definitions in place we are able to present the following Lemma.
\begin{lem}
Suppose $(x(s,t),u(s))$ is a $C^1$ S-curve for all $t>t_{s}$ with the points along the parametrization $s_0(t),s^*(t),s_1(t)$ satisfying the conditions prescribed in Definition \ref{EAC} at each time $t>t_s$. Then, within each overturned region respecivelty, the vertical line from $x(s_0(t),t)$ to $x(s_1(t),t)$ which yields $A_{Dif}(s_0(t),s^*(t),s_1(t),t)=0$ must move at the speed given by the Rankine-Hugoniot condition provided the shock remains isolated from other shocks. \label{lem1}
\end{lem}

\begin{proof}
By assumption,  $(x(s,t),u(s))$ is an S-curve and the provided functions $s_0(t),s^*(t)$ and $s_1(t)$ yields $A_{Dif}(s_0(t),s^*(t),s_1(t),t)=0$ for $t>t_s$. Also, by assumption that these functions are $C^1$, we have that $x(s_0(t),t)=x_0(s(t),0)+F'(u(s_0(t)))t$ is differentiable for all $t$. Utilizing the property that \\$x(s_0(t),t) = x(s^*(t),t) = x(s_1(t),t)$ for all time $t\geq t_s$, $A_{Dif}(s_0(t),s^*(t),s_1(t),t)=0$ provides us with the equation
\begin{equation}
x(s_0(t),t)(u(s_1(t))-u(s_0(t)))=\int_{s_{0}(t)}^{s_1(t)}{x(s(t),t)u'(s(t))-x_s(s(t),t)u(s(t)) \,\text{ds(t)}.}\label{pf2}
\end{equation}
Our desired quantity is $\frac{d}{dt}x(s_0(t),t)$, therefore we proceed to take the total derivative in time on both sides of equation (\ref{pf2}). Applying the chain rule, the fundamental theorem of calculus, properties of S-curves then cancelling terms from each side we arrive at
\begin{align}
(u(s_1(t))-u(s_0(t)))\frac{d}{dt}x(s_0(t)&,t)=x_s(s_0(t),t)s_0'(t)u(s_0(t))-x_s(s_1(t))s_1'(t)u(s_1(t))\nonumber\\
&+\int_{s_{0}(t)}^{s_1(t)}{\frac{\partial}{\partial t}\biggr(x(s(t),t)u'(s(t)-x_s(s(t),t)u(s(t))\biggr) \,\text{ds(t)}}\label{Lemma1EQ}
\end{align}
Equation (\ref{flow}) yields the relation $\frac{d}{dt}x(s(t),t)=x_s(s(t),t)s'(t)+F'(u(s(t))$. Using this and $\frac{d}{dt}x(s_0(t),t)=\frac{d}{dt}x(s_1(t),t)$, which is given by assumption, equation (\ref{Lemma1EQ}) becomes
\begin{equation}
2(u(s_1(t))-u(s_0(t))\frac{d}{dt}x(s_0(t),t)=F'(u(s_1(t))u(s_1(t))-F'(u(s_0(t))u(s_0(t)) + B(t),\label{AB}
\end{equation}
where 
\begin{equation}
 B(t)=\int_{s_{0}(t)}^{s_1(t)}{\frac{\partial}{\partial t}\biggr(x(s(t),t)u'(s(t))-x_s(s(t),t)u(s(t))\biggr) \,\text{ds(t)}}.\label{Bt}
\end{equation} 
We proceed to simplify (\ref{Bt}). We begin by utilizing $\frac{\partial}{\partial t}x(s(t),t)=F'(u(s(t))$, and that $\frac{\partial}{\partial t}u(s(t))=0$ to obtain

\begin{equation*}
B(t)=\int_{s_0(t)}^{s_1(t)}{F'(u(s(t)))u'(s(t))-u\frac{d}{ds(t)}\left(F'(u(s(t)))\right) ds(t)}
\end{equation*} 
Then, rewriting and applying the substitution $du=u'(s(t))ds(t)$ we get
 \begin{align*}
B(t)&=\int_{s_0(t)}^{s_1(t)}{F'(u(s(t)))u'(s(t))-u(s(t))F''(u(s(t)))u'(s(t)) ds(t)}\\
 &=\int_{u(s_0(t))}^{u(s_1(t))}{F'(u)-uF''(u) du}
 \end{align*}
From here, integration by parts and simplifying yields
\begin{equation*}
B(t)=2\biggr(F(u(s_1(t)))-F(u(s_0(t)))\biggr)+\left[u(s_0(t))F'(u(s_0(t)))-u(s_1(t))F'(u(s_1(t)))\right]
\end{equation*}
Plugging this into equation (\ref{AB}) we see that the second term of $B(t)$ cancels, then dividing through by $2[u(s_1(t))-u(s_0(t))]\neq0$, yields
%\begin{align*}
%2\frac{d}{dt}\big(x(s_0(t))\big)(u(s_0(t))-u(s_1(t))&=F'(u(s_0(t))u(s_0(t))-F'(u(s_1(t))u(s_1(t))-\\
%2\biggr(F(u(s_1(t)))+F(u(s_0(t)))\biggr)&+\left[u(s_1(t))F'(u(s_1(t)))-u(s_0(t))F'(u(s_0(t)))\right]
%%\end{align*}
%Cancelling terms and diving through by $2[u(s_0(t))-u(s_1(t))]$, gives us
\begin{equation}
\frac{d}{dt}\big(x(s_0(t),t)\big)=\frac{F(u(s_0(t))-F(u(s_1(t))}{u(s_0(t))-u(s_1(t))}
\end{equation}
as desired.
\end{proof}

\begin{lem}
Suppose we have a weak solution $u(x,t)$ of (\ref{PDE}), for $t\geq t_0$, containing exactly one shock at position $S(t)$ of the form
\begin{equation}u(x,t)=
\begin{cases}
u_L(x,t) \quad \text{for $x<S(t)$}\\ \label{Weak1}
u_R(x,t) \quad \text{for $x>S(t)$},
\end{cases}
\end{equation}
where $u_L(x,t)$ and $u_R(x,t)$ are smooth functions. Then, there exists $C^1$ functions $s_0(t), s^*(t)$ and $s_1(t)$ such that any smooth overturned S-curve joining the states $u_L$ to $u_R$ with zero signed area about the shock at $t=t_0$ satisfies\\ $A_{Dif}(s_0(t),s^*(t),s_1(t),t)=0$ for all $t\geq t_0$.\label{lem2}
\end{lem}
\begin{proof}

Parametrizing $u_L(x,t_0)$ and $u_R(x,t_0)$ we obtain $\left(x_L(s,t_0),u_L(s)\right)$ and \\
$\left(x_R(s,t_0),u_R(s)\right)$ respectively. We combine these into a single smooth S-curve\\
$\left(x(s,t_0),u(s))\right)$ which smoothly connects $\left(x(s_0,t_0),u(s_0)\right)$ to $\left(x(s_1,t_0),u(s_1)\right)$ such that $A_{Dif}(s_0,s^*,s_1,t_0)=0$ and $x(s_0,t_0)=x(s^*,t_0)=x(s_1,t_0)$. Such a curve can easily be constructed (e.g. with parametric cubics.) Let $s_0(t)$ and $s_1(t)$ be functions such that $\left(x(s_0(t),t),u(s_0(t))\right)$ and $\left(x(s_1(t),t),u(s_1(t))\right)$ lie on the top and bottom of the shock respectively for all $t>t_0$.  These functions must be $C^1$ since we have a smooth parametrization and the shock position is continuously differentiable if left isolated, see \cite{Daf}. Furthermore, since $x(s_0(t),t)$ and $x(s_1(t),t)$ lie on the shock, we have that $\frac{d}{dt}x(s_0(t),t)=\frac{d}{dt}x(s_1(t),t)=\frac{F(u(s_0(t))-F(u(s_1(t))}{u(s_0(t))-u(s_1(t))}$. We now prove that $A_{Dif}(s_0(t),s^*(t),s_1(t),t)=0$ for all $t\geq t_0$..

By assumption we have that $A_{Dif}(s_0(t_0),s^*(t_0),s_1(t_0),t_0)=0$, therefore we simply need to show that $\frac{d}{dt}A_{Dif}(s_0(t),s^*(t),s_1(t),t)=0$ for all $t>t_0$. Since we start with an S-curve at $t=t_0$, we have that the shape is preserved by our assumption that $F$ is uniformly convex and that the shock remains isolated. This means that the shock line will intersect the overturned curve exactly three times, satisfying $x(s_0(t),t)=x(s^*(t),t)=x(s_1(t),t)$ for all $t>t_0$. Differentiating $A_{Dif}$ in time yields 
\begin{align*}
\frac{d}{dt}A_{Dif}(s_0(t),s^*(t),s_1(t),t)&=x_s(s_0(t),t)s_0'(t)u(s_0(t))-x_s(s_1(t),t)s_1'(t)u(s_1(t))\\
&+\frac{d}{dt}x(s_1(t),t)\biggr((u(s_0(t))-u(s_1(t))\biggr)\\
&+\int_{s_{0}(t)}^{s_1(t)}{\frac{d}{dt}\biggr(x_s(s(t),t)u(s(t))-x(s(t),t)u'(s(t))\biggr) \,\text{ds}}
\end{align*}
Using again that $\frac{d}{dt}x(s_0(t),t)=x_s(s_0(t),0)s_0'(t)+F'(u(s_0(t)))=\frac{d}{dt}x(s_1(t),t)$ and other computations from Lemma 1 we reach
\begin{align*}
\frac{d}{dt}A_{Dif}&=2\frac{d}{dt}x(s_0(t),t)(u(s_0(t))-u(s_1(t)))+F'(u(s_1(t)))u(s_1(t))\\
&-F'(u(s_0(t)))u(s_0(t))+2\biggr(F(u(s_1(t)))-F(u(s_0(t)))\biggr)\\
&-F'(u(s_1(t)))u(s_1(t))+F'(u(s_0(t)))u(s_0(t))
\\
&=2\frac{d}{dt}x(s_0(t))(u(s_0(t))-u(s_1(t)))-2\biggr(F(u(s_1(t)))-F(u(s_0(t)))\biggr)
\end{align*}
Plugging in our assumption that the shock is moving according to the Rankine-Hugoniot condition we obtain 
\begin{equation}
\frac{d}{dt}A_{Dif}(s_0(t),s^*(t),s_1(t),t)=0.
\end{equation}
Therefore since $A_{Dif}(s_0(t_0),s^*(t_0),s_1(t_0),t_0)=0$, the shock line splits the S-curve into two regions with equal area for all $t\geq t_0$.
\end{proof}

\vspace{.2 cm}

Combining Lemma \ref{lem1} and Lemma \ref{lem2} we obtain the following result.

\begin{thm}[The Global Equivalence Theorem]\label{GET}
Suppose 
\begin{equation}u(x,t)=
\begin{cases}
u_L(x,t) \quad \text{for $x<S(t)$}\\
u_R(x,t) \quad \text{for $x>S(t)$},
\end{cases}
\end{equation}
is a weak solution of (\ref{PDE}) with $u_L(x,t)$ and $u_R(x,t)$ smooth functions.  Let $S(t)$ be the shock position at time $t\geq t_0$ which does not interact with any other shocks. Then $\dot{S}(t)=\frac{F(u_{max})-F(u_{min})}{u_{max}-u_{min}}$ if and only if there exist three differentiable functions $s_0(t),s^*(t)$ and $s_1(t)$ satisfying the conditions to make $( x(s,t),u(s) )$ an S-curve with zero signed area about $S(t)$ for $t\geq t_0$ with $\Big(x(s,t),u(s)\Big)=\Big(x,u_L(x,t))\Big)$, for $x<S(t)$ and 
$\Big(x(s,t),u(s)\Big)=\Big(x,u_R(x,t))\Big)$, for $x>S(t)$.

\end{thm}
\begin{proof}
$\Rightarrow$ follows from Lemma (\ref{lem2}) and $\Leftarrow$ follows from Lemma (\ref{lem1}).
\end{proof}
 \begin{rem}
 We note that we did not need to specify details about the portion of $(x(s,t),u(s))$ that was overturned, other than it being an S-curve. Therefore this result holds for any S-curve which satisfies the required regularity. This is advantageous numerically as this enables us to write a non-smooth Riemann problem into overturned S-curve, effectively turning a non-smooth problem into an equivalent smooth problem. Additionally, this provides flexibility from a particle management perspective, as we can keep the overturned region small by projecting down to a small S-curve when needed.
 \end{rem}
 
\section{The Generalized Equal Area Principle}\label{GEAP}

In this section we relax the requirement that the overturned portion of the curve is an S-curve. As mentioned above, when solving (\ref{PDE}) with uniformly convex flux function $F$, any isolated shock will result in an S-curve under the flow (\ref{flow}). This, however, is not the case when shocks collide.\\

Consider the initial condition illustrated in the left panel of Figure \ref{MultiShocks}. Until the two shocks collide, they move independently as S-curves, with their equal area projection agreeing with the Rankine-Hugoniot condition as shown in section \ref{global}.

\begin{figure}[!ht]
\begin{center}
\includegraphics[width=60mm,height=20mm]{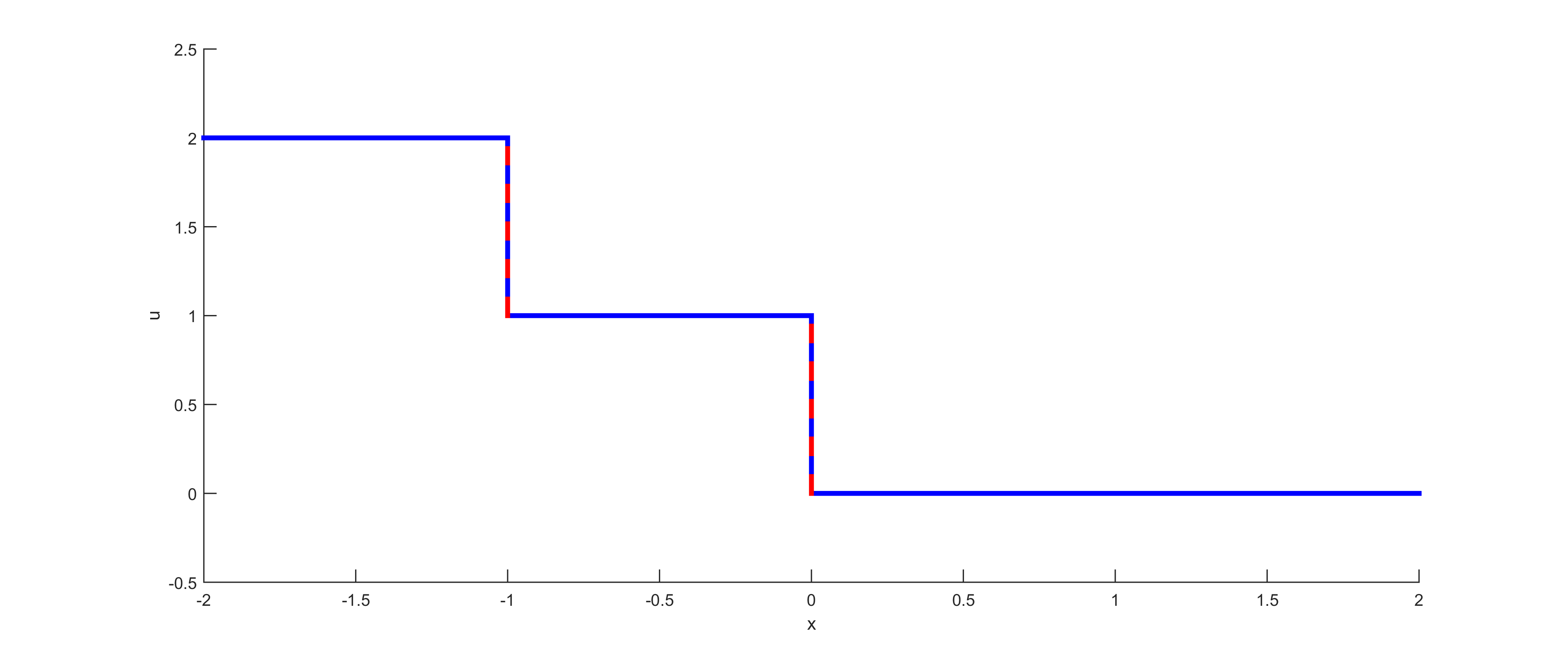}
\includegraphics[width=60mm,height=20mm]{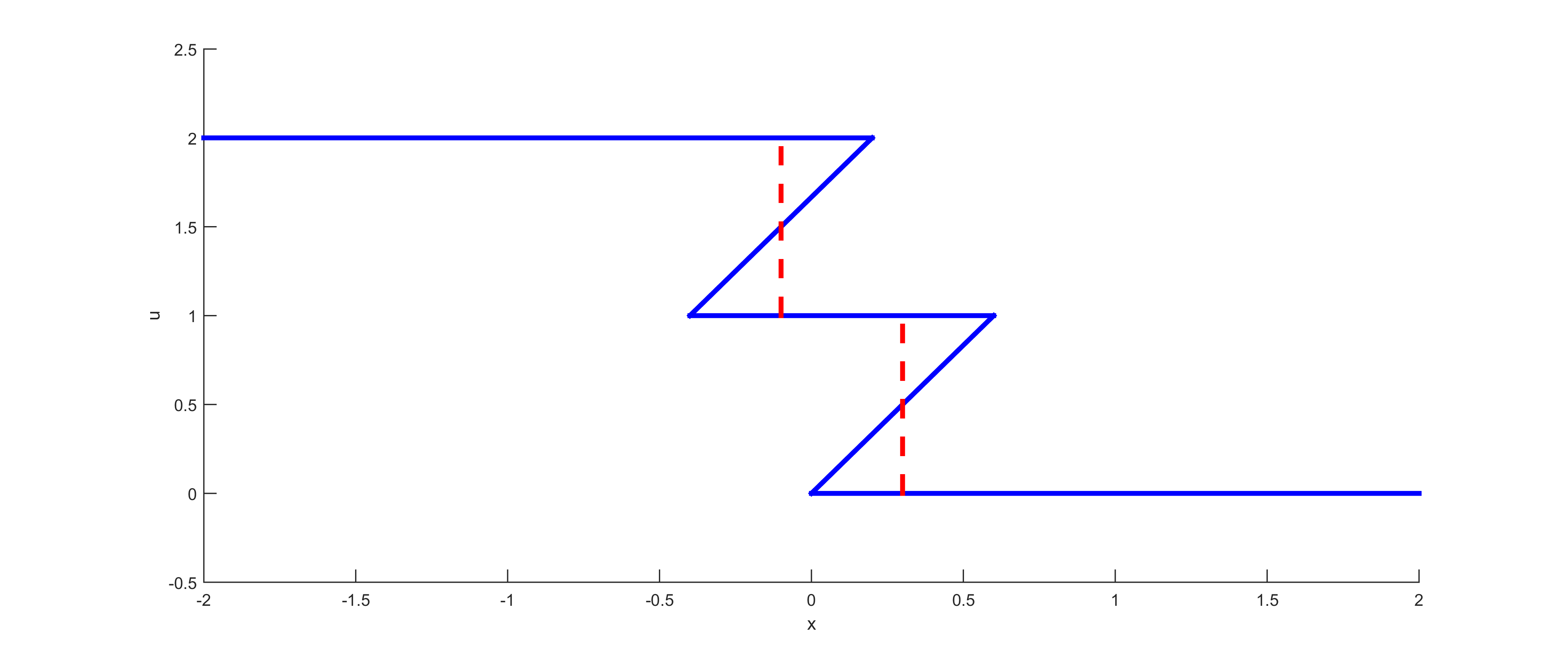}
\end{center}
\caption{Multiple Shocks with left panel showing $t=t_0$ and right panel $0<t<t^*$}
\label{MultiShocks}
\end{figure}

These shocks will propagate until a time $t^*$ when they collide, becoming one single shock and moving at a new speed given by the Rankine-Hugoniot condition. As illustrated in Figure \ref{MultiShocks2}, once collided the overturned region is no longer an S-curve, in fact, the shape will change depending on the number of collisions and the duration of the shock. Our main result, Theorem \ref{GET}, does not apply for $t>t^*$ as it can only be applied to the evolution of S-curves. We therefore need a theorem for more general multivalued regions. We call these equal area curves which we define below.

\begin{dfn}\label{EAC2}
Consider the set of parameter pairs defined by
\begin{equation*}
\mathcal{S}(t)=\{(s_0,s_1) \in \mathbb{R}^2 \big| x(s_0,t)=x(s_1,t),   \, \,  \text{such that} \int_{s_0}^{s_1}x_s(s,t)u(s)\text{ds}=0 \}.
\end{equation*}
In addition, we require that each pair in $\mathcal{S}(t)$ satisfies $u(s_0)>u(s_1)$, along with $\displaystyle \lim_{s\rightarrow s_0^{-}}x_s(s,t)>0$, and $\displaystyle \lim_{s\rightarrow s_1^{+}}x_s(s,t)>0$. If between each pair $(s_0,s_1) \, \in \mathcal{S}(t)$ we replace $(x(s,t),u(s))$ by a vertical line and as a result obtain a piecewise $C^1$ function with jump discontinuities, then $(x(s,t),u(s))$ is said to be an equal area curve. 
\end{dfn}

\begin{figure}[!ht]
\begin{center}
\includegraphics[width=60mm,height=20mm]{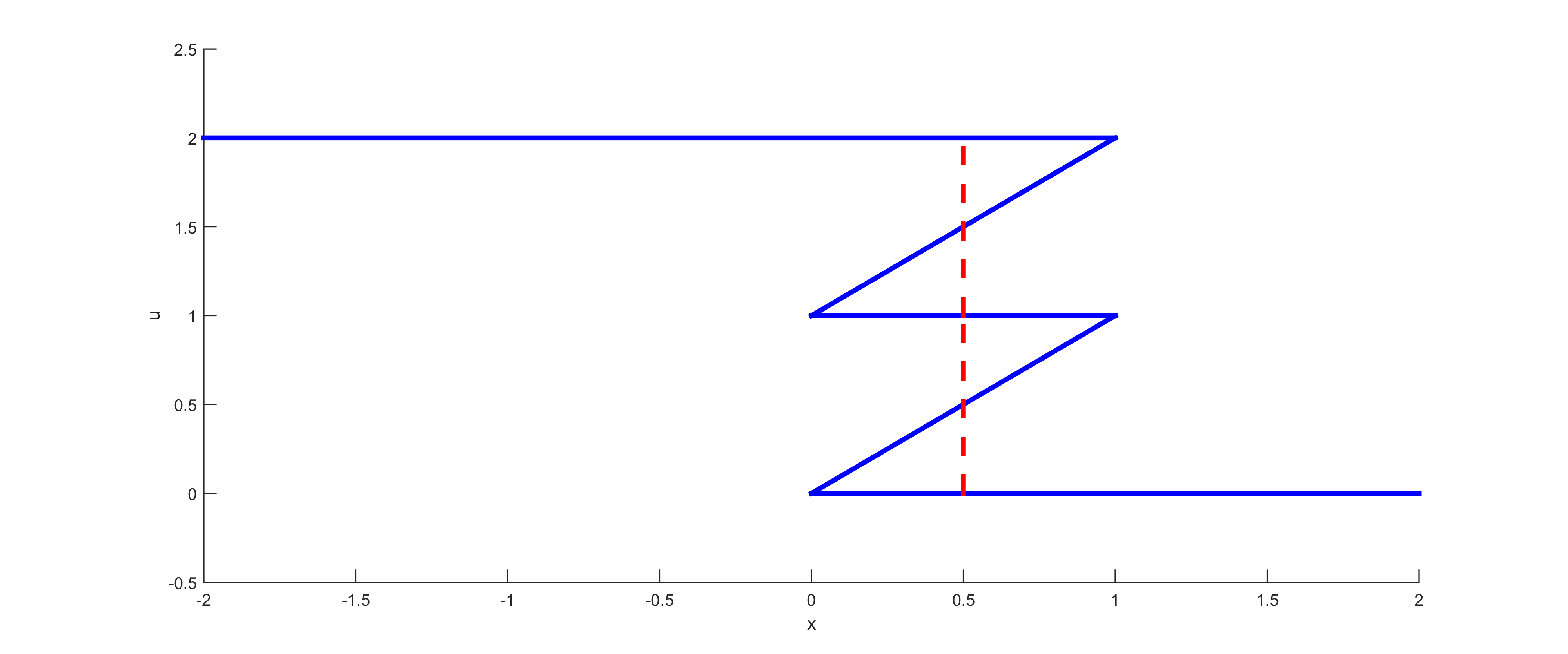}
\includegraphics[width=60mm,height=20mm]{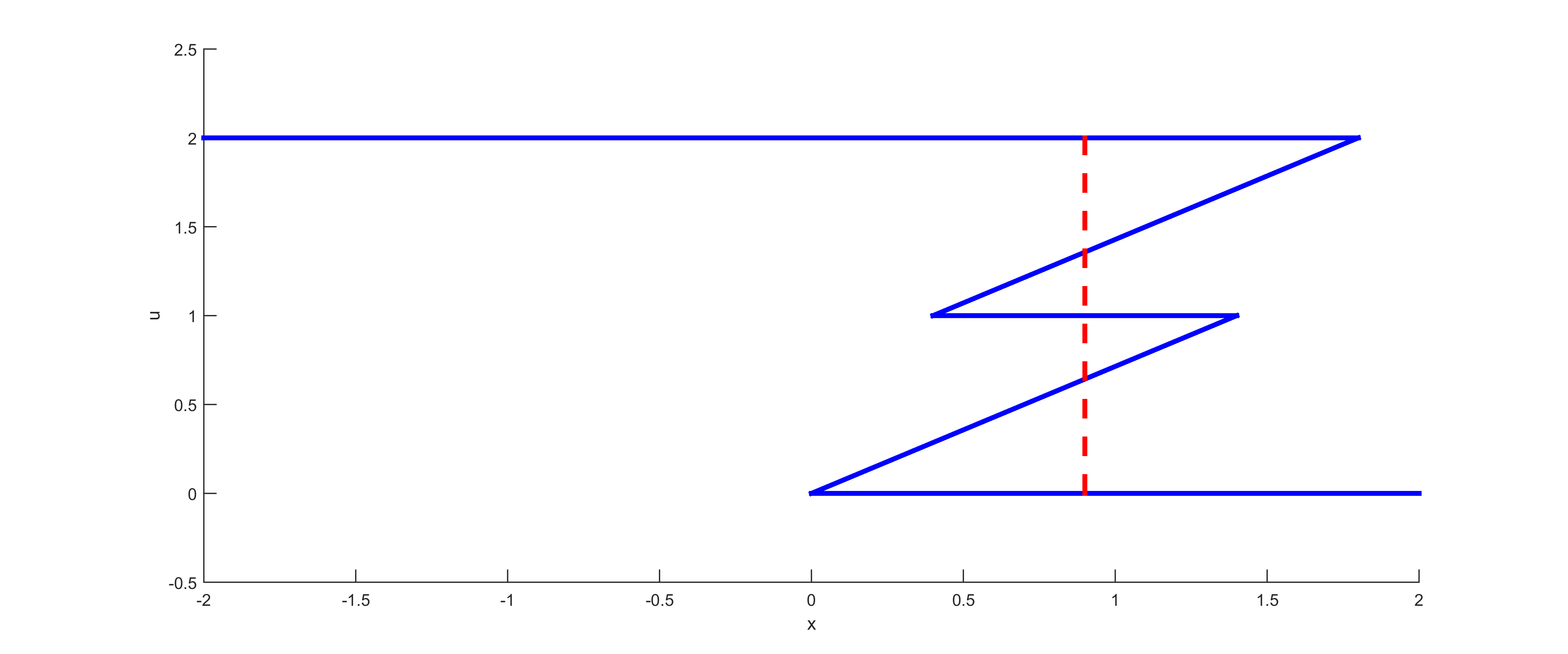}
\end{center}
\caption{Left panel shows $t=t*$ and right panel $t>t^*$}
\label{MultiShocks2}
\end{figure}

In order to prove that the equal area principle holds when shocks collide, we must first prove that the equal area principle holds for all equal area curves. Consider the following theorem.

\newpage 
\begin{thm}\label{EAPP} Let $(x(s,t),u(s))$ be a curve parametrized by $s$ defined by the flow $(x(s,t),u(s))=(s+F'(u(s))t,u(s))$. Suppose after some time $t$ the curve becomes multivalued, with average horizontal area about the multivalued region located at $A(t)$. Suppose that $A(t)$ intersects the weak solution obtained by the Rankine-Hugoniot condition along the top portion at $(x(s_0(t),t),u(s_0(t)))$, and along the bottom at $(x(s_1(t),t),u(s_1(t)))$. Then, 
\begin{equation}
\displaystyle A'(t)=\frac{F(u(s_0(t))-F(u(s_1(t))}{u(s_0(t))-u(s_1(t))},
\end{equation}
the same speed as given by the Rankine-Hugoniot condition.\end{thm}

\begin{proof}
Since we have no information about the shape of the multivalued portion of the curve $(x(s,t),u(s))$, we are unable to utilize Green's theorem, as the number of closed curves is unknown. One example would be a curve as in Figure \ref{Multi}. Notice that as this curve evolves its shape will change, rendering our Green's theorem approach obsolete. We therefore work directly with the parametric area given by 
\begin{equation}
A(t)=\int_{s_0(t)}^{s_1(t)}{u(s(t))x_s(s(t),t) ds(t)}\label{Area}.
\end{equation}

\begin{figure}[!ht]
\begin{center}
\includegraphics[width=60mm,height=20mm]{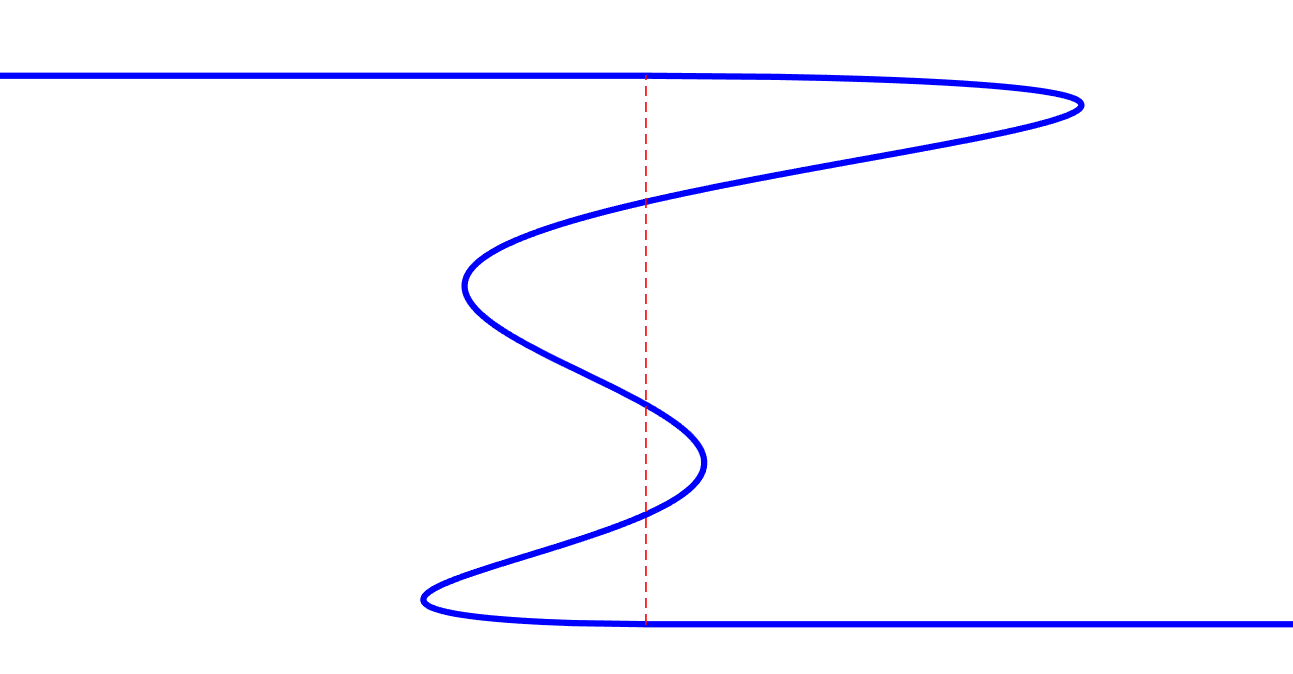}
\end{center}
\caption{ Example multivalued curve }
\label{Multi}
\end{figure}

Equation (\ref{Area}) tells us
\begin{align}
A'(t)=&\frac{d}{dt}\int_{s_0(t)}^{s_1(t)}{u(s(t))\partial_{s(t)}x(s(t),t) ds(t)} \nonumber\\
=&u(s(t))x_s(s(t),t)s'(t)\Big|_{s_0(t)}^{s_1(t)}+\int_{s_0(t)}^{s_1(t)}{\frac{\partial}{\partial t}u(s(t))x_s(s(t),t) ds(t)}\label{Dt}.
\end{align}
We simplify the two terms in equation (\ref{Dt}) separately. We begin with the first term.

\begin{equation}
u(s(t))x_s(s(t),t)s'(t)\Big|_{s_0(t)}^{s_1(t)}=u(s_1(t))\partial_{s(t)}x(s_1(t),t)s_1'(t)-u(s_0(t))x_s(s_0(t),t)s_0'(t)\label{eval}
\end{equation}

Using $x(s(t),t)=s(t)+F'(u(s(t)))t$, we obtain the relation
\begin{equation}
x_s(s(t),t)s'(t)=\frac{d}{dt}x(s(t),t)-F'(u(s(t))).\label{rel}
\end{equation}

Plugging in (\ref{rel}) into (\ref{eval}), and using $x(s_0(t),t)=x(s_1(t),t)$, we obtain

\begin{align}
u(s(t))x_s(s(t),t)s'(t)\Big|_{s_0(t)}^{s_1(t)}&=\frac{d}{dt}x(s_0(t),t)(u(s_1(t))-u(s_0(t)))\nonumber\\
&+F'(u(s_0(t)))u(s_0(t))-F'(u(s_1(t)))u(s_1(t))\label{term1}.
\end{align}

Turning to the second term in equation (\ref{Dt}), we obtain the following simplification
\begin{align*}
\int_{s_0(t)}^{s_1(t)}{\frac{\partial}{\partial t}u(s(t))x_s(s(t),t) ds(t)}&=\int_{s_0(t)}^{s_1(t)}{\frac{\partial}{\partial t}u(s(t))(1+F''(u(s(t)))u'(s(t))t) ds(t)}\\
&=\int_{s_0(t)}^{s_1(t)}{F''(u(s(t)))u'(s(t))u(s(t) ds(t)}.
\end{align*}
After applying a substitution and integration by parts we obtain
\begin{align}
\int_{s_0(t)}^{s_1(t)}{\frac{\partial}{\partial t}u(s(t))\partial_{s(t)}x(s(t),t) ds(t)}&=F'(u(s(t))u(s(t))\Big|_{s_0(t)}^{s_1(t)}-\int_{s_0(t)}^{s_1(t)}{F'(u(s(t))ds(t).}\label{term2}
\end{align}

Plugging (\ref{term1}) and (\ref{term2}) into (\ref{Dt}) yields
\begin{align}
A'(t)&=\frac{d}{dt}x(s_0(t),t)(u(s_1(t))-u(s_0(t)))
+F'(u(s_0(t)))u(s_0(t))\nonumber\\
&-F'(u(s_1(t)))u(s_1(t))+F'(u(s_1(t)))u(s_1(t))-F'(u(s_0(t)))u(s_0(t))\\
&+F(u(s_0(t)))-F(u(s_1(t))).\nonumber\\
&=\frac{d}{dt}x(s_0(t),t)(u(s_1(t))-u(s_0(t)))+F(u(s_0(t)))-F(u(s_1(t))).\label{result}
\end{align} 

From here we see that $A'(t)=0$ if and only if\\ $\displaystyle \frac{d}{dt}x(s_0(t),t)=\frac{F(u(s_0(t))-F(u(s_1(t))}{u(s_0(t))-u(s_1(t))}$ under the given assumptions, which completes the proof.
\end{proof}

%\begin{rem}
%We note that as in Theorem 1 we did not need to specify any details about the curve $(x(s,t),u(s))$ other than it being an equal area curve. 
%\end{rem}

Now that we have the equal area principle for any overturned equal area curve, we can prove the following Corollary on shock collisions.

\begin{cor} \label{ShockInt} Suppose $(x(s,t),u(s))$ is an equal area curve associated with the entropy satisfying weak solution of (\ref{PDE}). In addition, suppose that $(x(s,t),u(s))$ has exactly two shocks which collide at $t=t^*$. Then, at $t=t^*$, $(x(s,t^*),u(s))$ is an equal area curve with exactly one shock at the point of collision which propagates at Rankine-Hugoniot speed.
\end{cor}

 \begin{proof}
As in Theorem 1, parametrize the first shock with top $s_{10}(t)$ and bottom $s_{11}(t)$ and the second with top $s_{20}(t)$ and bottom $s_{21}(t)$. Both of these are $C^1$ functions on $[0,t^*]$. Note that $s_{11}(t^*)=s_{20}(t^*)=s^*\in(0,1)$. Now, chopping the initial data $(s,g(s))$ from $s\in[0,1]$ down to $s\in[0,s^*]$, we have that $(x(s,t),u(s))$ has exactly one shock. This results in an S-curve up to $t=t^*$. Repeat this argument on the second portion $s\in[s^*,1]$. This results in stacked S-curves at $t=t^*$, with $x(s_{10}(t^*),t^*)=x(s_{11}(t^*),t^*)=x(s_{20}(t^*),t^*)=x(s_{21}(t^*),t^*)$, both having equal area on the right and left. The resulting curve is therefore an equal area curve. Thus, for $t>t^*$, Theorem 2 tells us that the equal area projection agrees with the Rankine-Hugoniot condition.
 \end{proof}

With proofs for Theorem \ref{GET} and Theorem \ref{EAPP} along with Corollary \ref{ShockInt}, we have that equation (\ref{flow}) can always be used to obtain weak solutions of the Cauchy problem (\ref{PDE}) using the equal area principle. Therefore, if we are able to obtain an accurate parametric polynomial representation of (\ref{flow}), then upon an appropriate equal area projection we can obtain an accurate representation of the desired weak solution. In the next section we show how the ingredients required to generate such parametric polynomial interpolants are readily available. 

\section{Numerical Approach}\label{Numerics}
 %The main objective of this paper was to develop a rigorous framework for numerical methods which utilize the equal area principle. In particular, we set out to show that we can obtain weak solution of (\ref{PDE}) by applying an equal area projection of the parametric curve (\ref{flow}). In this section we show how this enables us to solve for weak solutions of (\ref{PDE}) by working directly with a parametric polynomial representation of (\ref{flow}).  Although specifics are left for future work, we show how all the tools required to build such methods are readily available. 

We begin by recalling the implications of the above analysis. We have shown, for any problem within the scope of (\ref{PDE}), that the weak solutions can be obtained from the parametric curve $(x(s,t),u(s))=\left(s + F'(g(s))t,g(s)\right)$, by employing the equal area principle. That is, if we have a parametric polynomial approximation of (\ref{flow}), then, as parametric polynomials can be integrated exactly, we can obtain an accurate representation of the weak solution through an area preserving projection. We now provide a brief discussion of how one uses (\ref{flow}) to create the required parametric polynomial interpolants.

Suppose we partition the initial curve $(x,g(x))$ into $n$ subintervals at points $x_i$, for $i=0,..,n$.  To obtain an accurate representation of the parametric curve between $x_i$ and $x_{i+1}$ at time $t$ one requires endpoint values, derivative information, and in this case, the parametric area. We show how each of these ingredients are easily obtained from (\ref{PDE}), without any accumulation of error. 

The endpoints of each interpolating parametric polynomial are easily obtained by evaluating the flow at time t. Specifically, the curve initially interpolating from $(x_i,g(x_i))$ to $(x_{i+1},g(x_{i+1})$, would, after time $t$, interpolate from $\left(x_{i}+F(g(x_{i})t,g(x_{i})\right)$ to $\left(x_{i+1}+F(g(x_{i+1})t,g(x_{i+1})\right)$ , which is a simple evaluation. Tangent information at time $t$ at each node $\left(x(x_i,t),g(x_i)\right)$ is obtained by evaluating the parametric function $\left(1+F''(g(x_i))g'(x_i)t,g'(x_i)\right)$, with higher order information available if $F$ and $g$ possess sufficient regularity. And finally, the parametric area under the curve from $\left(x_{i}+F(g(x_{i})t,g(x_{i})\right)$ to $\left(x_{i+1}+F(g(x_{i+1})t,g(x_{i+1})\right)$ is given by 
\begin{align}
\int_{x_i}^{x_{i+1}}g(s)\left(1+F''(g(s))g'(s)t\right)ds&=\int_{x_i}^{x_{i+1}}g(s)ds+\int_{x_i}^{x_{i+1}}F''(g(s))g'(s)g(s)t ds,\nonumber\\
&=\int_{x_i}^{x_{i+1}}g(s)ds+t\int_{g(x_i)}^{g(x_{i+1})}F''(u)u du\nonumber\\
&=\int_{x_i}^{x_{i+1}}g(s)ds+t\left(F'(u)u-F(u)\right)\Big|_{g(x_i)}^{g(x_{i+1})}\label{Area11}
\end{align}
which can be calculated analytically, provided we can find an expression for $ \int{g(x) dx} $.  We note that only the second term depends on time, therefore in cases where we are unable to compute $\int{g(x)}dx$ exactly, a precise numerical integration can be used at $t=0$, then as time evolves we simply update the second term. Therefore, the quantities required to produce accurate interpolating polynomials are easily initialized and are inexpensive to update.

\begin{rem}
We will not discuss details of the parametric interpolation itself as it is not within the scope of the current paper. For example, one can match data from a function, matching tangents, area and function value, but the resulting parametric polynomial may be multivalued, an outcome that would cripple this approach. Therefore we leave this discussion for future work where the required precision and care can be taken.
\end{rem}

 We note that if such a scheme can be properly  implemented, then it will be a conservative scheme, as each parametric interpolant will exactly preserve the parametric area. We also note that in the case where (\ref{flow}) can be exactly represented in the chosen space of parametric polynomials, then this framework will produce the exact solution, up to machine error. This is the case in the example done below with Burgers' equation.

%\begin{rem}
%In the appendix we discuss how using the above Corollary %enables us to determine the shock position as a simple %polynomial root finding problem. 
%\end{rem}

\subsection{Numerical Results}
We now apply the framework presented above to the problem solved in section 2 of this paper,
\begin{equation*}
\begin{cases}
u_t+\left(\frac{u^2}{2}\right)_{x}=0,\quad \text{on $\mathbb{R}\times(0,\infty)$}\\       
u(x,0)=g(x)\quad \text{on $\mathbb{R}\times\{0\}$},
\end{cases}
\end{equation*}
where the initial condition $g(x)$ is defined by
\begin{equation*}g(x)=
\begin{cases}
0 \quad \text{for $x<0$}\\ 
x \quad \text{for $0\leq x \leq 1$}\\
0 \quad \text{for $x>1$}.
\end{cases}
\end{equation*}

For this example we only have to interpolate the line from $(0,0)$ to $(1,1)$, and the vertical line from $(1,0)$ to $(1,1)$, as $F'(0)=0$ when working with Burgers' equation. For this particular example we will obtain the desired accuracy using simple linear interpolation. We begin by noting that the first line segment at $t=0$ is exactly given by $P_1(s,0)=(s,s),$ for $s\in [0,1]$. Therefore, applying the flow (\ref{flow}) we obtain $P_1(s,t)=(s+st,s)$. Similarly, the vertical line is given by $P_2(s,0)=(1,s)$, for $s\in[0,1]$, which yields $P_2(s,t)=(1+st,s)$, for $s\in[0,1]$. Using the technique described in the appendix, we employ polynomial interpolation to $P_1(s,t)$ and $P_2(s,t)$. We derive the following set of conditions from $P_1(s,t)$ in order to generate the Hermite cubic $H_1(x,t)$ between $x=0$, given by $s=0$, and $x=1+t$, given by $s=1$

\begin{align*}
(0,H_1(0,t))&=P_1(0,t)=(0,0),\\ (1+t,H_1(1+t,t))&=P_1(1,t)=(1+t,1) \\
\frac{\partial}{\partial x}H_1(0,t)&=\frac{1}{1+t} , \, \, \, \frac{\partial}{\partial x}H_1(1+t,t)=\frac{1}{1+t}. 
\end{align*}
From here we obtain $H_1(x,t)=\frac{x}{1+t}$ and applying the same technique to $H_2(x,t)$ and obtain $H_2(x,t)=\frac{x-1}{t}$. Applying the root finding discussed in the appendix yields the quadratic
\begin{equation}
X^2-(1+t)=0
\end{equation}
whose roots for $X\in[0,1+t]$ yield the shock position at time $t$. A sample simulation is illustrated in Figure \ref{Simulation1} with error shown in Figure \ref{ErrorPlot}.

\begin{rem}
We note that when solving Burgers' equation, the space of parametric polynomials maps to itself under the flow (\ref{flow}). For example, suppose the initial condition $g(s)=as^2+bs+c$, then after time $t$ the flow yields the parametric curve $P(s,t)=(s+(g(s))t,g(s))=(ats^2+(bt+1)s+ct,g(s))$, which has remained in the space of parametric quadratics. Therefore, if we are able to exactly interpolate $(s,g(s))$, then we retain an exact representation of $(s+g(s)t,g(s))$ for all time. This explains how we are able to obtain machine precision in the example presented here. More care is required when solving conservation laws other than Burgers' equation for this reason.
\end{rem}

\begin{figure}[!ht]
\begin{center}
\includegraphics[width=63mm,height=30mm]{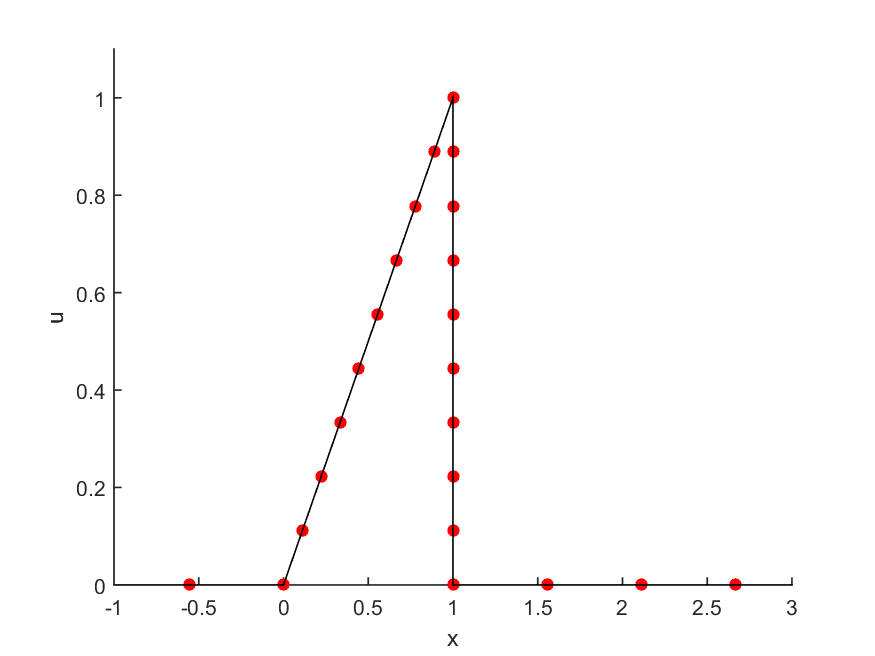}
\includegraphics[width=63mm,height=30mm]{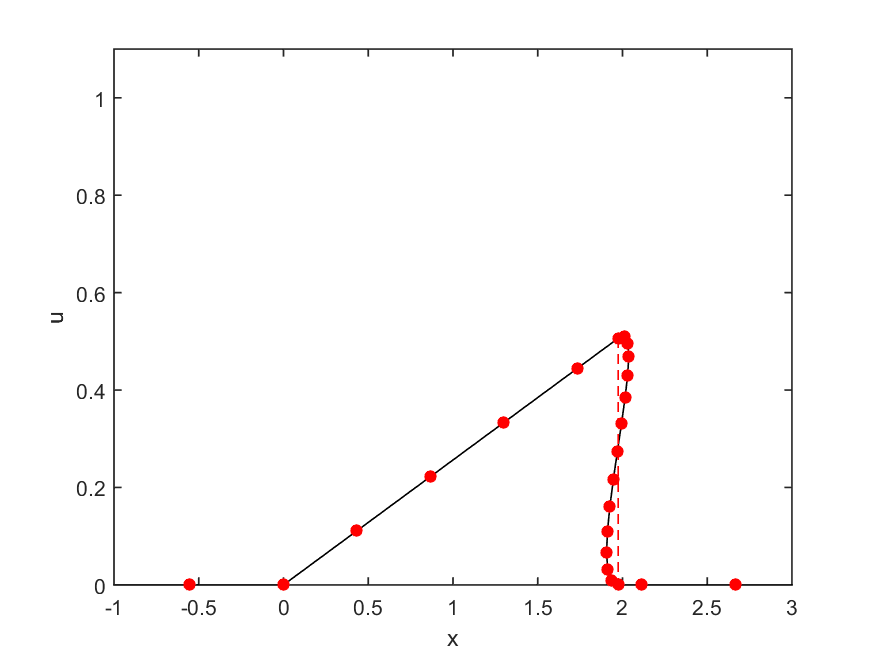}
\end{center}
\caption{Simulation of method at $t=0$ and $t=3$}
\label{Simulation1}
\end{figure}

In Figure \ref{ErrorPlot} we show the error in shock position after a fixed amount of time $t^*$ as we vary the time step $\Delta t$. We achieve machine precision as expected, with variances in error coming from round-off errors and machine errors in the root finding problem. Of course, we do not expect this precision for an arbitrary Cauchy problem of this type, but leave applications to other flux functions for future study.

\begin{figure}[!ht]
\begin{center}
\includegraphics[width=120mm,height=60mm]{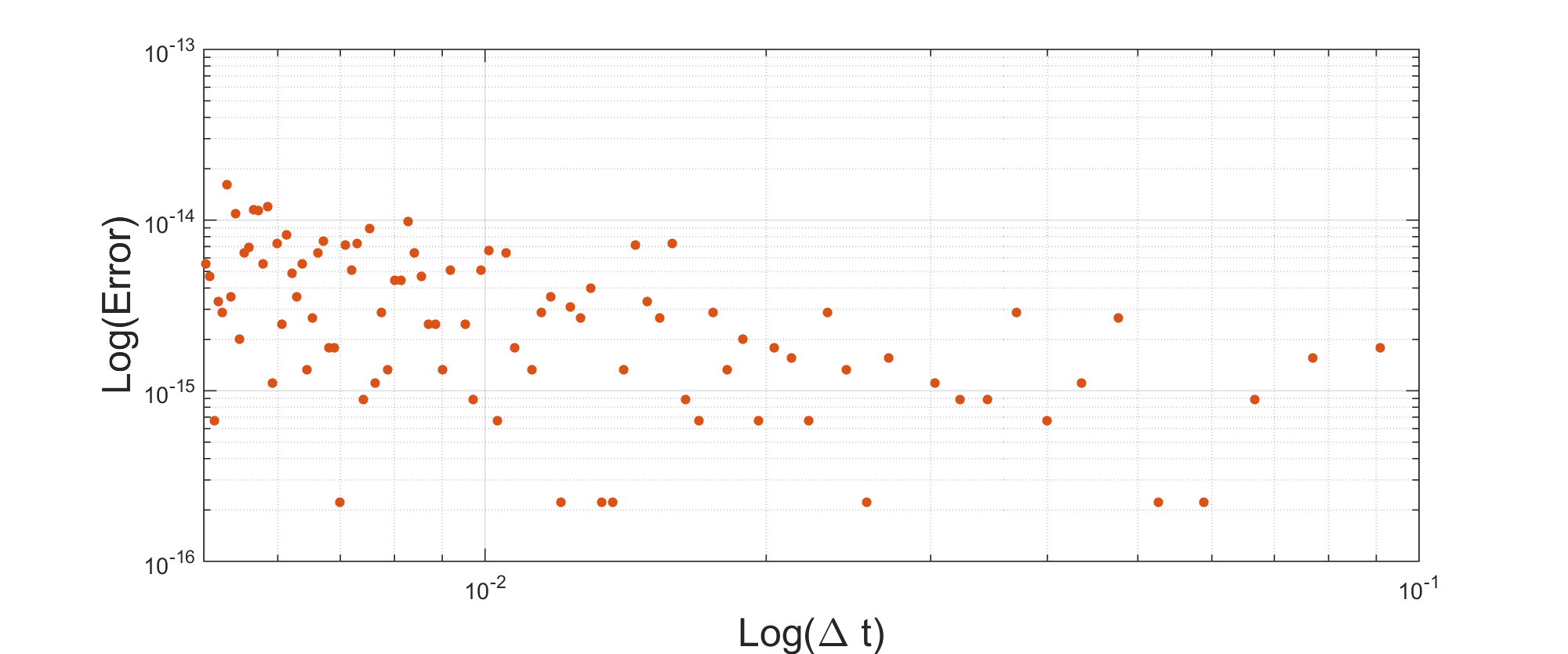}
\end{center}
\caption{Error in the shock position at $t=10$}
\label{ErrorPlot}
\end{figure}

We began with the above example because it avoids many of the subtleties of parametric interpolation which we are not discussing in this paper. We now present a more complex example and show how each of the quantities described above can be easily be computed.

Consider the Cauchy problem

\begin{equation}
\begin{cases}
u_t+\left(\frac{u^2}{2}\right)_{x}=0,\quad \text{on $\mathbb{R}\times(0,\infty)$}\\ 
\\      
u(x,0)=g(x)=\begin{cases}1+\arctan(-x)\quad \text{on $x\in [-10,10]\}$},\\
0\quad \text{on $x\in (-\infty,-10)\cup(10,\infty)$}.
\end{cases}
\end{cases}
\end{equation}

We begin by sampling $n$ points, $(x_0,u_0), \,(x_1,u_1), \dots, (x_{n-1},u_{n-1})$, from the initial curve, including both vertical lines and the curve $(x,g(x))$.  More specifically we have
\begin{align*}
&(x_i,u_i)=\left(-10,\frac{i}{k-1}(1+\arctan(-10))\right), \, \text{ $i=0,\dots, k-1$}\\
&(x_i,u_i)=\left(-10+20\frac{(i-k)}{m},1+\arctan\left(-10+20\frac{(i-k)}{m}\right)\right), \, \text{ $i=k,\dots,k+m$}\\
&(x_i,u_i)=\left(10,\frac{(n-1)-i}{(n-1)-(k+m+1)}(1+\arctan(10))\right), \, \text{ $i=k+m+1,\dots, n-1$}.
\end{align*}

Letting each point move under the characteristic flow for $\Delta t$ yields a new set of points $(x_i+u_i\Delta t,u_i)$, where $x_i$ and $u_i$ from each portion of the curve are given above. Our numerical solution is obtained by extracting data from the exact solution curve to construct parametric polynomials through our sample points. For $i=0,\dots,k-1$ we have the exact curve given by\\ $\left(-10+s(1+\arctan(-10)\Delta t,s(1+\arctan(-10)\right)$, for $s\in[0,1]$, which is simply a line. This is the same for the second vertical line. The main part of the curve has exact solution given by
\begin{equation}
\Phi_B(x,g(x),\Delta t)=\left(x+(1+\arctan\left(-x\right))\Delta t,1+\arctan\left(-x\right)\right), \quad \text{for $x\in[-10,10]$}
\end{equation}
For each pair of sample points, $(x_i + u_i\Delta t,u_i)$ and $(x_{i+1} + u_{i+1}\Delta t,u_{i+1})$, for $i=k,\dots,k+m-1$, we generate a parametric polynomial interpolant, $(X_{i}(s),U_{i}(s))$, for $s\in[0,1]$, matching the tangents given by,
\begin{align*}
(D_sX_{i}(0),D_sU_{i}(0))&=\left(1+\frac{\Delta t}{1+x_i^2},\frac{-1}{1+x_i^2}\right), \text{and}\\
(D_sX_{i}(1),D_sU_{i}(1))&=\left(1+\frac{\Delta t}{1+x_{i+1}^2},\frac{-1}{1+x_{i+1}^2}\right).
\end{align*} 
If higher order interpolation is desired, we can compute higher order derivatives using the same approach. As we want a conservative scheme, we require the interpolants to exactly preserve the parametric area. Therefore, for each $i=k,\dots,k+m-1$, we require
\begin{equation}
\int_0^1{U_{i}(s)D_sX_{i}(s) ds}=\int_{x_{i}}^{x_{i+1}}{(1+\arctan(-x))\left(1-\frac{\Delta t}{1+x^2}\right) dx},
\end{equation}
which can be simplified using equation (\ref{Area11}) and integrated exactly.
%partitioning the initial data into $n$ Hermite polynomials, which in this example we will assume to be fourth order. Therefore each interpolating polynomial will be of the form
%\begin{equation}
%H_i(s)=a_is^4+b_is^3+c_is^2+d_is+e_i, \quad \text{for $s\in[0,1]$},
%\end{equation} with $-10=s_1<s_2<\dots<s_{n-1}<s_n=10$. We obtain the correct choices for $a_i,b_i,c_i$ and $d_i$ by requiring that each polynomial matches function value, derivatives at each endpoint and area. We always compute the derivatives from the interior of each polynomial, for example, $\displaystyle H_1'(0)=\lim_{x\rightarrow-10^+}g'(x)$. The final step is to interpolate the vertical lines at $x=-10$ and $x=10$, meaning we take
%\begin{align}
%&(s,H_0(s))=(-10,(1+\arctan(-10)s), \quad \text{for $s\in[0,1]$, and}\\
%&(s,H_{n+1}(s))=(10,(1+\arctan(10)s), \quad \text{for $s\in[0,1]$.}\nonumber
%\end{align}
%We note that in order to obtain the desired convergence results as we increase the number of polynomial interpolants, the same refinement will need to be done along the lines of discontinuity. As we will see shortly however, this is not required when solving Burgers equation.

\section{Discussion}\label{Disc}
In this paper we set out to provide a rigorous framework for which new numerical methods can be derived to solve hyperbolic conservation laws.  In particular we wanted to prove that numerically we can find weak solution of (\ref{PDE}) by working entirely with a parametric polynomial representation of (\ref{flow}). To reach this goal, stronger analytical results for the equal area principle were required.  Using techniques similar to those found in \cite{Equal} and \cite{Whit}, we developed the basis for the global equivalence theorem, proving an equivalence between the Rankine-Hugoniot condition and the equal area principle for isolated shocks. To deal with the case when shocks interact, we required a more general result than global equivalence theorem, which we refer to as the generalized equal area principle. This result has many applications, one of which is the corollary presented in the appendix, proving that we can work directly with the shock line instead of the multivalued curve given by (\ref{flow}).  Most importantly, upon proving Theorem 2, we concluded that the flow (\ref{flow}) can be used to obtain weak solutions of (\ref{PDE}) by employing the equal area principle for any piecewise smooth initial condition $g(x)$. This yields the desired result that an accurate representation of (\ref{flow}) yields an accurate representation of the desired weak solution. Next we showed the ease at which we can compute the required quantities to produce an accurate parametric polynomial interpolation of (\ref{flow}).

  %This result, which we call the generalized equal area principle, enables the use of any such equal area curve to find the weak solution of the Cauchy problem (\ref{PDE}).    This is useful for numerics as it allowed us to prove thatIt is exactly this flexibility which delivers the proof of our Corollary, tying our rigorous efforts to the numerical analysis of the problem.  

The example concluding section \ref{Numerics} delivers a proof of concept for our proposed numerical framework. This example is intentionally simple as more work is needed before the numerical theory is complete. For example, further study is required regarding the projection back onto the space of parametric polynomials after each time step. We showed that if, for example, a parametric area preserving Hermite interpolation is used then the method exactly preserves derivatives and function values along with parametric area. However in the setting of parametric polynomials, simply providing this data is not enough alone to uniquely determine a parametric polynomial. Instead, additional information regarding the parametrization speed is required. This investigation is not well suited for this paper, but will appear in future work.

% addition to looking for an optimal parametric polynomial basis, a full analysis of accuracy and stability is required for these new methods to become viable. The next step would be to extend such methods to include multiple shocks and perhaps relax the conditions on the flux function $F$.

We believe the framework presented in this work provides a solid foundation for new numerical methods to be developed. Although we have only discussed 1-dimensional conservation laws, we are optimistic that similar ideas can be applied to higher dimensional problems and systems of conservation laws. Such topics and directions are the focus of our future work.

\section{Appendix}
We begin with an important Corollary of the Global Equivalence Theorem which ties the numerical approach to the rigorous theory presented above. 
 \begin{cor}\label{Cor1}
 Let $u(x,t)$ be a weak solution of (\ref{PDE}) containing a shock at $S(t)$ for $t\geq t_0$. Then, the equal area principle applied to any S-curve with zero signed area about the shock position $S(t)$ gives the same shock location as flowing the weak solution itself and applying the equal area principle on the overturned piecewise smooth curve. \end{cor}
 
We begin our discussion by referencing Figure \ref{CorPlot}. Plotted on the left is a weak solution as described in Corollary \ref{Cor1}. It can either be viewed as the step function or as the S-curve, with their equivalence being a result of Theorem \ref{GET}. Taking all points on the left curve $(x(s,0),u(s))$ and flowing them under Burgers' equation for one second, $(x(s,0),u(s)) \longrightarrow (x(s,0)+u(s),u(s))=(x(s,1),u(s))$, gives us the second plot with the true shock position represented by the vertical dashed line. Corollary \ref{Cor1} states that the shaded regions, A and B from Figure \ref{CorPlot}, do not need to be included in the computation to find the shock location. Instead, we can simply find the vertical line which makes the area of the two triangles, $T_1$ and $T_2$, equal.  To generate a proof of this we require that the characteristic map (\ref{CharFlow1}) is conservative.

\begin{figure}[!ht]
\begin{center}
\includegraphics[width=63mm,height=60mm]{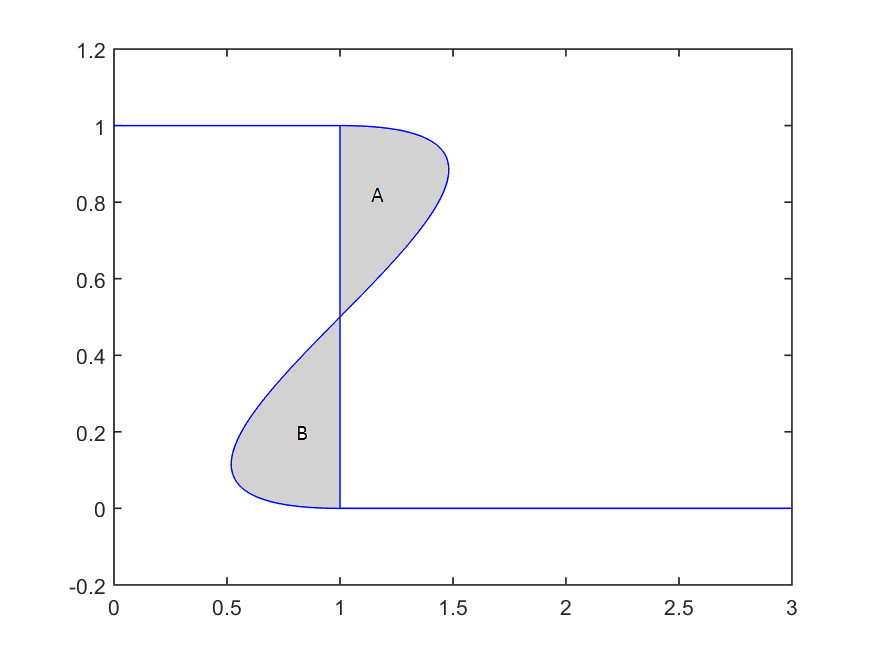}
\includegraphics[width=63mm,height=60mm]{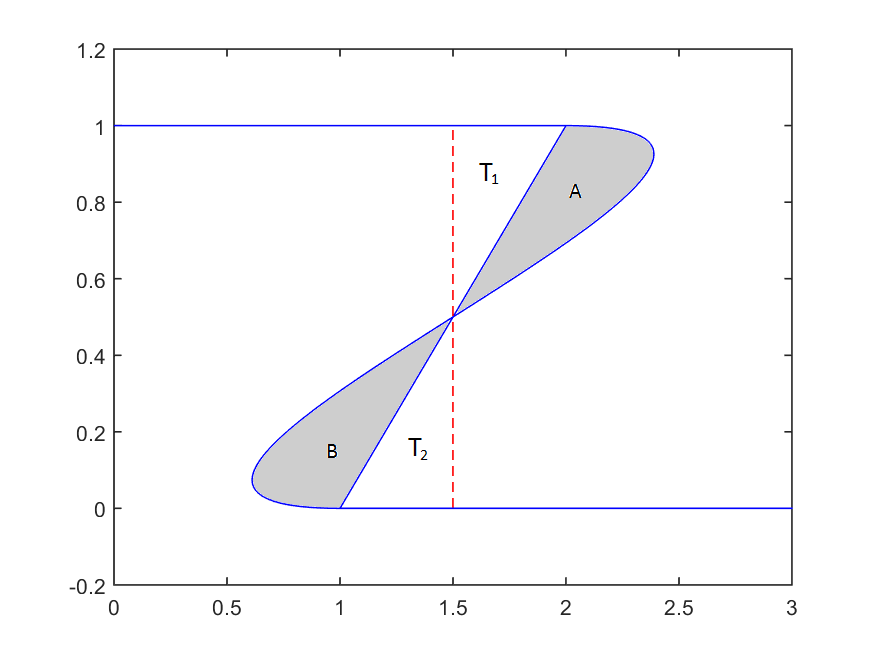}
\end{center}
\caption{Corollary}
\label{CorPlot}
\end{figure} 

Consider the flow map $\Phi (x(s,0),u(s),t)=(x(s,0) + F'(u(s))t,u(s))$. This has jacobian
\begin{equation} J(\Phi)=
\begin{vmatrix} \dfrac{\partial \Phi_1}{\partial x} & \dfrac{\partial \Phi_1}{\partial u} \\ & \\ \dfrac{\partial \Phi_2}{\partial x} & \dfrac{\partial \Phi_2}{\partial u} \end{vmatrix}=\begin{vmatrix} 1 & F''(u(s))t \\ & \\ 0 & 1 \end{vmatrix},
\end{equation}
and therefore, $\det(J(\Phi))=1$, implying that $\Phi$ is a conservative map. Therefore, any closed curves under the flow will conserve area. Applying this result to Figure \ref{CorPlot} proves that the $Area(A)$ in the left plot equals $Area(A)$ in the right plot. The same holds for B. Given that each of $T_1,\, T_2, \, A$ and $B$ are pairwise disjoint and that we started with a zero signed area S-curve, we know $Area(A)+Area(T_1)=Area(B)+Area(T_2)$. Using the additional property, $Area(A)=Area(B)$, we have that, indeed, $Area(T_1)=Area(T_2)$. It appears we have a proof of the corollary, however, it is not always true that that these sets will be disjoint for an arbitrary zero signed area S-curve. To see this we turn our attention to Figure \ref{CorPlot2}. 
\begin{figure}[!ht]
\begin{center}
\includegraphics[width=63mm,height=60mm]{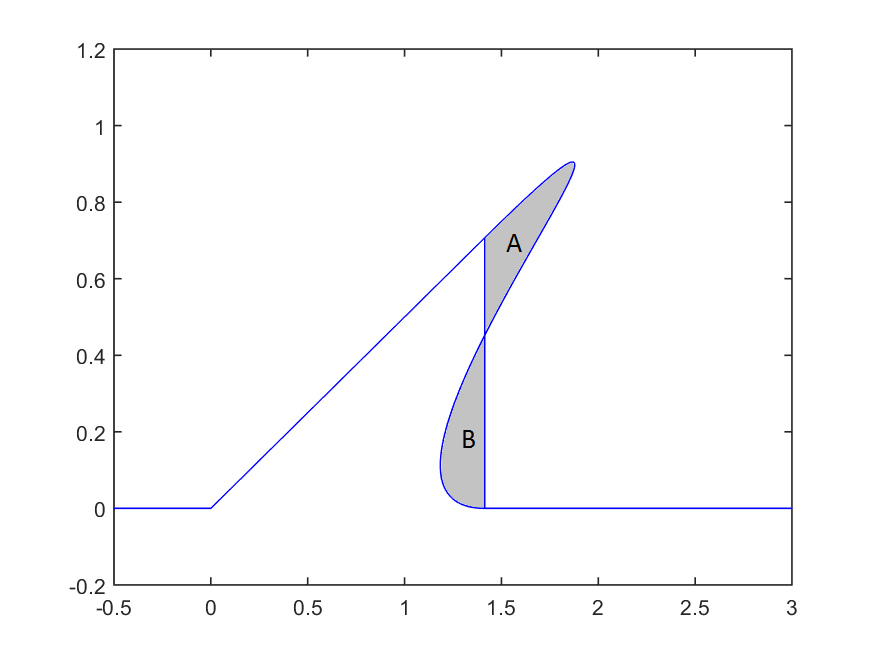}
\includegraphics[width=63mm,height=60mm]{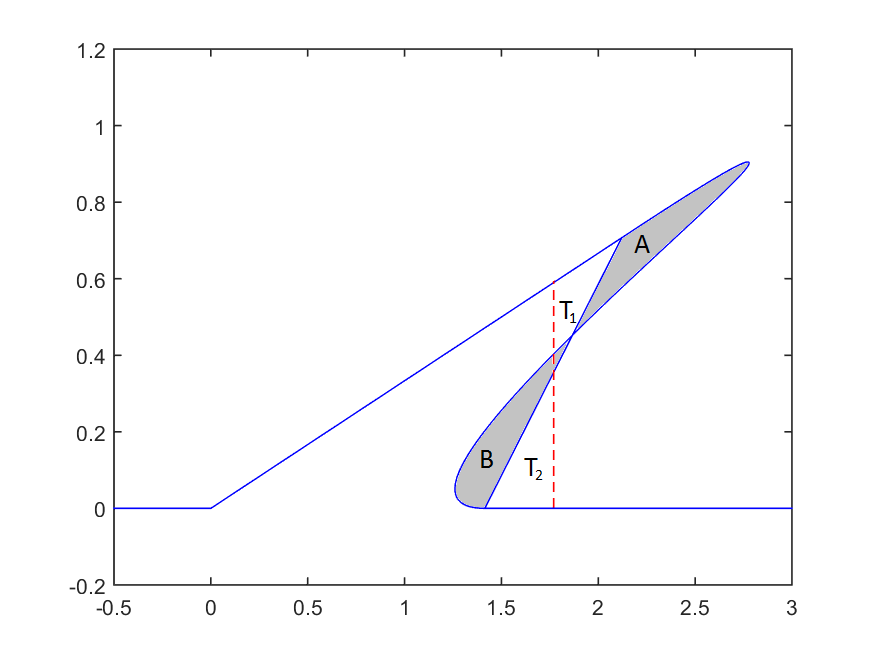}
\end{center}
\caption{An unclear application of the Corollary}
\label{CorPlot2}
\end{figure} 
Initially we are in the same situation as before, however, it is clear by the second plot that $B\cap T_1\neq\emptyset$, therefore the result that $Area(T_1)=Area(T_2)$ is not immediate. To resolve this, we simply rely on the property that the global equivalence theorem holds for any S-curve with the same signed area. We therefore replace the S-curve with a new S-curve, matching the same data, which satisfies the property that the new regions $T_1,\, T_2, \, A^*$ and $B^*$ are pairwise disjoint.  An illustration of such a curve is shown in Figure \ref{CorSaved}. %This situation is rectified by the following argument. Recall that the slanted blue line, $H(s,t)$, is the result of mapping the vertical shock line, $H(s,0)$, under the characteristic flow (\ref{CharFlow1}), where $H(0,0)$ is the bottom of the shock and $H(1,0)$ is the top of the shock in the left plot.  Clearly we can find a smooth overturned curve, intersecting $H(s,t)$ exactly three times, connecting $H(1,t)$ to $H(0,t)$ which has zero signed area about $H(s,t)$ and intersects $H(s,t)$ exactly where the current shock crosses $H(s,t)$. An illustration of such a curve is shown in Figure \ref{CorSaved}.

\begin{figure}[!ht]
\begin{center}
\includegraphics[width=80mm,height=60mm]{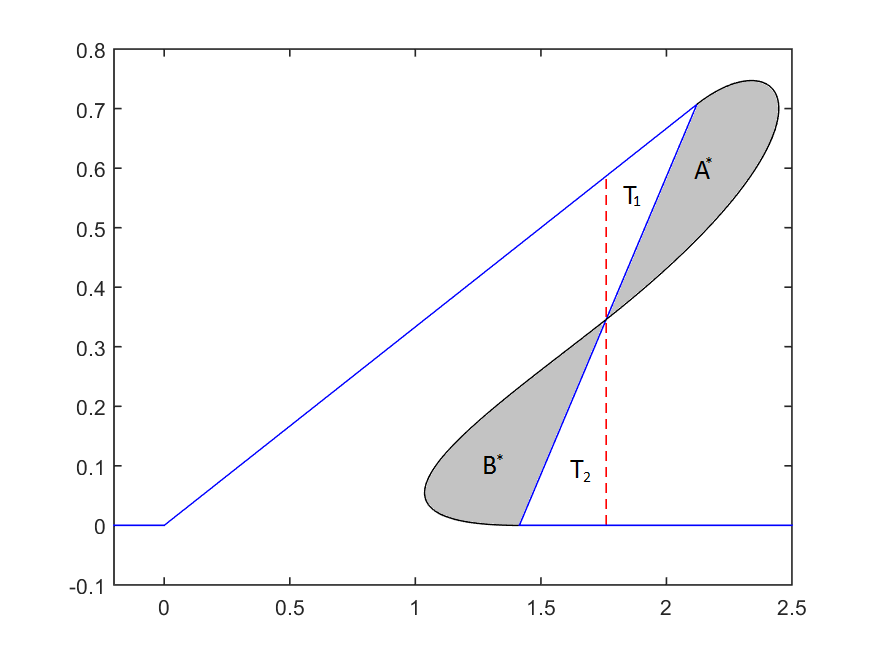}
\end{center}
\caption{The optimal curve}
\label{CorSaved}
\end{figure} 
Using the new curve shown in Figure \ref{CorSaved} does not change our solution as it has the same signed area about the slanted line, so reversing time to $t=0$ yields an equivalent S-curve about the shock. Using this new curve, with each of $T_1, \, T_2, \, A^*$ and $B^*$ disjoint, tells us $Area(T_1)=Area(T_2)$, which proves our corollary. Next we show how this result immediately enables us to obtain the shock position by finding solutions to a root finding problem.

Corollary \ref{Cor1} provides us with the ability to work directly with the weak solution, which seemingly distinguishes the numerical approach from the rigorous theory. At first glance, this may seem to discount much of our early motivation, but this is an incorrect observation. The method provided by the corollary formulates the shock motion as a polynomial root finding problem. However, this will only be accurate if the polynomial approximation of the upper and lower curves are accurate which requires accurate data for $u_L$ and $u_R$. Additionally, the application of Corollary \ref{Cor1} relies on the initial shock position being accurate. Both requirements are addressed by allowing the curve to overturn initially. This provides a setting where we can accurately find the initial shock location and produce a piecewise smooth polynomial approximation of the weak solution. We now sketch how one obtains the shock location through a root finding problem.

Suppose we have a weak solution $u(x,t)$ as shown in the left plot of Figure \ref{Simulation1}, with $u(x,t)=u_L(x,t)$ for $x<S(t)$ and $u(x,t)=u_R(x,t)$ for $x>S(t)$.  For simplicity we enforce the additional assumption that this shock remains isolated. Flowing this curve for $\Delta t$ seconds we end up with an overturned curve as shown in Figure \ref{CorPlot3}. Under the assumption that no other shocks form near $S(t)$, we have $\Phi(u_L(x,t),\Delta t)$ and $\Phi(u_R(x,t),\Delta t)$ remain functions, given by $f_{1}(x)=\Phi(u_L(x,t),\Delta t)$ and $f_{2}(x)=\Phi(u_R(x,t),\Delta t)$. Additionally we have that $\Phi(H(s,t),\Delta t)$, the shock line mapped forward $\Delta t$ seconds, will also be a function, since $F$ is taken to be uniformly convex. We therefore can interpolate $f_1(x), \, f_2(x)$ and $H(x)$ with $n^{th}$ order Hermite polynomials, given by $P(f_1(x)), \, P(f_2(x))$ and $P(H(x))$, as an example, since  we maintain access to derivative information at every point. Using that the shock location must be somewhere in the overturned region, we label $(a_1,a_2)=\Phi(H(0,t),\Delta t)$ and $(b_1,b_2)=\Phi(H(1,t),\Delta t)$. This  implies the new shock must be contained between $x=a_1$ and $x=b_1$. Assuming the shock location is given by $x=S(t+\Delta t)$, we have the region $B$ from Figure \ref{CorPlot3} has area given by
\begin{align}
Area(B)&=\int_{a_1}^{S(t+\Delta t)}{H(x)-f_2(x)\,dx}, \quad \text{and similarly,}\\
Area(A)&=\int_{S(t+\Delta t)}^{b_1}{f_1(x)-H(x)\,dx}.
\end{align} 

\begin{figure}[!ht]
\begin{center}
\includegraphics[width=120mm,height=60mm]{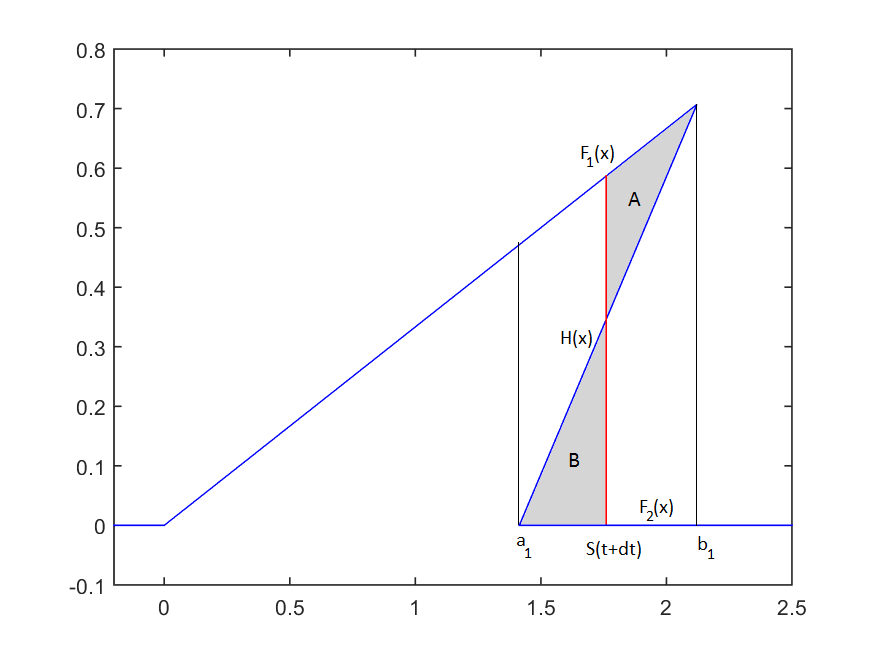}
\end{center}
\caption{Numerical approach}
\label{CorPlot3}
\end{figure}

In reality, we do not have the location of the shock $S(t+\Delta t)$ and only have access to our polynomial approximations of $f_1, \, f_2$ and $H$, but knowing that we want $Area(A)=Area(B)$, we derive a polynomial equation for the shock location $S(t+\Delta t)$, given as a root of the $(n+1)^{st}$ order polynomial
\begin{equation}
\int_{a_1}^{S(t+\Delta t)}{P(H(x))-P(f_2(x))\,dx}-\int_{S(t+\Delta t)}^{b_1}{P(f_1(x))-P(H(x))\,dx}=0, \label{PolyShock}
\end{equation}
It is clear by the formulation of the problem that there is a unique solution $S(t+\Delta t) \in (a_1,b_1)$ to equation (\ref{PolyShock}).  Therefore, the Corollary enables us to determine the shock position by finding the unique root of a polynomial in an open interval, an extremely fast operation. Once the shock location is found, we remove the overturned part of the curve and repeat this process with the new weak solution.

\section*{Acknowledgements} We are grateful to Marc Laforest, who contributed many hours of engaging and thought provoking discussion on conservation laws and related topics.
\newpage
\bibliographystyle{siam}
\bibliography{references}
\end{document}